\documentclass[a4paper,10pt]{article}

\usepackage[round]{natbib}   

\usepackage{mathrsfs}
\usepackage{amsfonts, epsfig, amsmath, amssymb, color, mathabx}

\usepackage{environ}
\usepackage{mdframed}

\usepackage{amsthm}
\usepackage{thmtools}
\usepackage{mathtools}
\usepackage{enumitem}
\usepackage{scalerel}
\usepackage{dsfont} 
\usepackage{breqn}
\declaretheorem[style=remark]{remark}
\declaretheorem[name=Theorem,within=section]{theorem}
\declaretheorem[name=Lemma,numberlike=theorem]{lemma}
\declaretheorem[name=Proposition,numberlike=theorem]{proposition}

\declaretheorem[name=Definition,numberlike=theorem]{definition}
\usepackage{comment}
\usepackage{todonotes}

\usepackage{wrapfig}
\usepackage{tikz}
\usepackage{pgfplots}
\usepgfplotslibrary{fillbetween}
\usetikzlibrary{patterns}
\usepackage[justification=centering]{subcaption}
\pgfplotsset{every axis/.append style={tick label style={/pgf/number format/fixed},font=\scriptsize,ylabel near ticks,xlabel near ticks,grid=major}}

\newcommand{\diff}[1]{\, \mathrm{d}#1}

\newcommand{\set}[1]{\left\{#1\right\}}

\newcommand{\expected}[2][]{
  \ifthenelse{\equal{#1}{}}
  {\mathbb{E}\!\left\{#2\right\}}
  {\operatorname{E}_{#1}\!\left\{#2\right\}}
}

\newcommand{\be}{\begin{equation}}
\newcommand{\ee}{\end{equation}}
\newcommand{\ben}{\begin{equation*}}
\newcommand{\een}{\end{equation*}}

\newcommand{\ba}{\begin{equation}\begin{aligned}}
\newcommand{\ea}{\end{aligned}\end{equation}}
\newcommand{\ban}{\begin{equation*}\begin{aligned}}
\newcommand{\ean}{\end{aligned}\end{equation*}}   

\newcommand{\bN}{\mathbb{N}}
\newcommand{\bR}{\mathbb{R}}

\newcommand{\ex}{\mathrm{e}}
\newcommand{\di}{\mathrm{d}}         
\renewcommand{\P}{\mathbf{P}}  

\newcommand {\e}{\varepsilon}

\newcommand{\rD}{\mathscr{D}}

\newcommand{\rF}{\mathscr{F}}

\DeclareMathOperator{\sign}{sign}
\DeclareMathOperator{\Law}{Law}

\let\oldmarginpar\marginpar
\renewcommand{\marginpar}[1]{\oldmarginpar{\scriptsize\texttt{\color{red}{#1}}}}  
\begin{document}

\title{Heterogeneous diffusion process with power-law nonlinearity}



\author{Jorge E.\ Cardona\thanks{\texttt{jorge@cardona.co}} \  and Ilya Pavlyukevich\thanks{Institute of Mathematics,
Friedrich Schiller University Jena, Inselplatz 5,
07743 Jena, Germany; \texttt{ilya.pavlyukevich@uni-jena.de}}
}

\maketitle

\begin{abstract}
In this paper, we study solutions of the heterogeneous diffusion process with power-law nonlinearity
governed by the stochastic differential equation
$\di X_t= |X_t|^\alpha\,\di B_t + \alpha\lambda |X_t|^{2\alpha-1}\sign(X_t)\,\di t$, where
$\alpha\in (0,1)$ and
$\lambda\in[0,1]$. The parameter $\alpha$ controls the nonlinear power-law profile of
the diffusion coefficient, while the parameter $\lambda$ specifies the
interpretation of the stochastic integral in the pre-equation $\dot X=|X|^\alpha\dot B$. We demonstrate that the solutions
of this equation can be represented as nonlinear transformations of a skew
Bessel process with dimension $\delta \in \mathbb{R}$.

\end{abstract}


\smallskip


\numberwithin{equation}{section}


\section{Introduction}

In the physical literature, a Brownian motion process $B$ is often referred to as \emph{free diffusion}, while the
solution of the SDE
\ba
\label{e:hd-phys}
\dot X&=\sigma(X)\dot B
\ea
is called a \emph{heterogeneous diffusion}, with $\sigma$ representing the diffusion coefficient.

From a mathematical perspective, this corresponds to a one-dimensional SDE with multiplicative noise.
However, from the physical point of view, equation \eqref{e:hd-phys}
does not completely specify the process $X$. Instead, it serves as a ``pre-equation''
that must be complemented by an interpretation rule in order to define an actual stochastic process
(see \cite{vanKampen81}).

The interpretation rule determines the choice of stochastic integral.

In the classical It\^o framework, the SDE \eqref{e:hd-phys} is interpreted as the stochastic integral equation
\ba
\label{e:Ito}
X_t = x + \int_0^t \sigma(X_s) \,\di B_s
&:= x+ \lim_{n\to\infty}
\sum_{k} \sigma(X_{t_{k}^n\wedge t})
(B_{t_{k+1}^n\wedge t}-B_{t_{k}^n\wedge t})\\
\ea
where $\{t^n_k,k\in\bN_0\}_{n\in\bN}$ denotes a sequence of partitions of the half-line $[0,\infty)$ with
mesh size tending uniformly to zero, i.e.,
$\lim_{n\to\infty}\max_k|t^n_{k+1}-t^n_k|=0$,
The limit in \eqref{e:Ito} is taken in the ucp (uniform convergence on compacts in probability) sense,
see, for example,  Chapter II in \cite{Protter-04}.
It is the characteristic feature of the It\^o stochastic integral, that the integrand $\sigma(X)$
is always evaluated at the left endpoint of each partition interval.

Different interpretations of multiplicative noise can be introduced by means of an
\emph{interpretation parameter} $\lambda \in [0,1]$.

The corresponding $\lambda$-stochastic integral is defined as the ucp-limit
\ba
\int_0^t \sigma(X_s) \circ_\lambda \di B_s
&:= \lim_{n\to\infty}
\sum_{k} (\lambda \sigma(X_{t_{k+1}^n\wedge t})
-(1-\lambda) \sigma(X_{t_{k}^n\wedge t}))
(B_{t_{k+1}^n\wedge t}-B_{t_{k}^n\wedge t}),
\ea
which can be rewritten as
\ba
\int_0^t \sigma(X_s) \circ_\lambda \di B_s
&= \int_0^t \sigma(X_s)\,\di B_s + \lambda [\sigma(X),B]_t,
\ea
where $[\sigma(X),B]$ denotes the bracket process
\ba
{}[\sigma(X_t),B]_t
:=\lim_{n\to\infty} \sum_{k} (\sigma(X_{t_{k+1}^n\wedge t})-\sigma(X_{t_{k}^n\wedge t}))(B_{t_{k+1}^n\wedge t}-B_{t_{k}^n\wedge t}),
\ea
provided, the limit exists.
If $\sigma$ is sufficiently smooth, for instance, $\sigma\in C^2$, the bracket process can be expressed
explicitly as a Lebesgue integral,
\ba
\label{e:sigmabracket}
{} [\sigma(X),B]_t=\int_0^t \sigma(X_s)\sigma'(X_s)\,\di s.
\ea
In this case, the interpretation parameter $\lambda$ contributes
an additional \emph{noise-induced drift} term $\lambda \sigma(X)\sigma'(X)$,
so that the $\lambda$-interpretation of \eqref{e:hd-phys} coincides with the It\^o SDE
\ba
\label{e:Leb}
X_t=x+ \int_0^t \sigma(X_s)\,\di B_s+\lambda \int_0^t \sigma(X_s)\sigma'(X_s)\,\di s.
\ea
The best-known special case is the Stratonovich interpretation with $\lambda = \frac{1}{2}$; see Section V.5 of \cite{Protter-04}.
Another important case is $\lambda = 1$, referred to in the physics literature as
the Hänggi--Klimontovich or kinetic interpretation; see \cite{Sokolov-10}.

This paper focuses on a particular heterogeneous diffusion with the irregular diffusion coefficient
$\sigma(x)=|x|^\alpha$, $\alpha\in (0,1)$. Such a diffusion was first considered in the It\^o setting
by \cite{Girsanov-62} as an example of an SDE without the uniqueness property. More recently, it was
examined in \cite{CherstvyCM-13} in the context of the Stratonovich SDE
\ba
\label{e:stra_br}
\dot X=|X|^\alpha\circ\, \dot B,
\ea
corresponding to the interpretation parameter $\lambda = \tfrac{1}{2}$.
In that work, the authors approached the problem at a formal, physical level
of rigour and derived the probability density function for a symmetric solution.
Further, in Chapter 3.3 of \cite{Heidernaetsch-15} and in \cite{kazakevivcius2016influence},
explicit densities of non-negative
solutions for general $\lambda$-interpretations were presented.

The Stratonovich heterogeneous diffusion \eqref{e:stra_br} was studied rigorously
using stochastic calculus in \cite{PavShe-20} and \cite{PavShe2025} for $\alpha \in (-1,1)$.
It was shown that the equation is underdetermined and admits infinitely
many strong solutions. Restricting attention to the class of homogeneous strong
Markov solutions that spend zero time at $0$, one obtains solutions
as nonlinear transformations of a skew Brownian motion.
Specifically, for $\alpha \in (0,1)$ and any $\theta \in [-1,1]$, the process
\ba
\label{e:Xstrat}
X_t= |(1-\alpha) Z^\theta_t|^{\frac{1}{1-\alpha}}\sign(Z^\theta_t)
\ea
is a strong solution to \eqref{e:stra_br}
where $Z^\theta$ is the unique strong solution to the SDE
\ba
Z^\theta_t = z + B_t + \theta L^0_t(Z^\theta),\quad z=  \frac{1}{1-\alpha} |x|^{1-\alpha}\sign(x),
\ea
where $L^0$ denotes a symmetric semimartingale local time at $0$.
For $\alpha \in (-1,0]$, only the symmetric solution with $\theta = 0$
is possible (see Theorem 4.5 in \cite{PavShe-20}). Moreover, \cite{PavShe2025} established that the
symmetric solution with $\theta = 0$, first identified in \cite{CherstvyCM-13},
is physically relevant in the sense that it arises as the
unique solution of a stochastic selection problem.

For $\alpha \in (-1,0]$, only the symmetric solution with $\theta = 0$ is possible (see Theorem 4.5 in
\cite{PavShe-20}). Moreover, \cite{PavShe2025} established that the symmetric solution with $\theta = 0$,
first identified in \cite{CherstvyCM-13}, is physically relevant in the sense that it arises as the
unique solution of a stochastic selection problem.
\ba
\label{eq:main-simple}
X_t=x+ \int_0^t |X_s|^\alpha \,\di B_s + \alpha\lambda \int_0^t  |X_s|^{2\alpha-1}\operatorname{sign}(X_s) \,\di s.
\ea
which is the formal analogue of \eqref{e:Leb} with $\sigma(x)=|x|^\alpha$.
It should be emphasized that while the identity \eqref{e:sigmabracket} holds for smooth $\sigma$, the corresponding relation
\ba
{}[|X|^{\alpha},B]_t=
\alpha\int_0^t  |X|^{2\alpha-1}\operatorname{sign}(X_s) \,\di s
\ea
remains an open question in this singular setting.
Partial progress in this direction has been obtained in \cite{PavShe-20} and \cite{PavShe2025}.

It is clear that for any initial value $x\neq 0$, the SDE \eqref{eq:main-simple} admits a
unique local solution. Since both the diffusion coefficient and the drift grow sublinearly at infinity,
this local solution does not blow up and can be uniquely extended up to the first hitting time of the origin,
\ba
\tau_0=\begin{cases}
        \inf\{t\in [0,\infty)\colon X_t=0\},\\
        +\infty,\quad X_t\neq 0,\quad t\in[0,\infty).
       \end{cases}
\ea

Our analysis of the SDE \eqref{eq:main-simple} is based on a nonlinear transformation of the solution
X
X and its reduction to a Bessel process. Such transformations have been employed for heterogeneous diffusions in \cite[Chapter 3.3]{Heidernaetsch-15} and \cite{carr2006jump} in the context of financial mathematics.

The intuition is as follows. Consider the nonlinear, one-to-one mapping
\ba
H_\alpha(x)
= \frac{1}{1-\alpha}|x|^{1-\alpha} \operatorname{sign}(x),\quad x\in\bR,
\ea
with the inverse
\ba
H^{-1}_\alpha(z) = ((1- \alpha) |z|)^{\frac{1}{1-\alpha}} \operatorname{sign}(z),\quad z\in\bR,
\ea
which appeared earlier in \eqref{e:Xstrat}. Applying It\^o's formula to the local solution $X$
of \eqref{eq:main-simple} with initial value
$x\neq 0$, we find that the process
$Z:=H_\alpha(X)$ satisfies
\ba
Z_t&=H_\alpha(x) +\frac{\delta-1}{2}\int_0^t \frac{\di s}{Z_s} + B_t,\quad t\in[0,\tau_0),\\
\delta& =\delta_{\alpha,\lambda}= \frac{1 - 2\alpha (1 - \lambda)}{1 - \alpha}\in\bR.
\ea
For $x\in (0,\infty)$,
the process $Z$ is a well-known $\delta$-dimensional Bessel process.
If the initial point is negative, then the process $\widetilde Z=-H_\alpha(X)$ is a Bessel process satisfying
\ba
\widetilde Z_t=|H_\alpha(x)| +\frac{\delta-1}{2}\int_0^t \frac{\di s}{\widetilde Z_t} + \widetilde B_t,
\quad t\in[0,\tau_0),
\ea
with $\widetilde B=-B$.
Since we consider solutions on the time interval
$[0,\tau_0)$, both $X$ and $Z$ are well-defined and unique up to the first hitting time of the origin.

Summarizing this intuition, we see that the solution
$X$ of \eqref{eq:main-simple}, if it exists, behaves like a nonlinear transformation of a (possibly negative)
Bessel process away from
zero. Consequently, the key question is to describe the possible behaviors of $Z$
upon hitting zero.

In this paper, we focus on solutions that spend zero time at zero and are strong Markov
processes, assuming that a continuation beyond the first hitting of zero is possible.

\begin{definition}\label{def:weak-sln-simple}
A weak solution to \eqref{eq:main-simple} is a pair
$(X, B)$ of adapted continuous processes on
a stochastic basis $(\Omega, \rF, \mathbb{F},\P)$
such that
  \begin{enumerate}
  \item $B$ is a standard Brownian motion on $(\Omega, \rF, \mathbb{F}, \P)$;
  \item for any $t\in[0,\infty)$, the integrals $\int_0^t |X_s|^\alpha\, \di B_s$
  and $\int_0^t |X_s|^{2\alpha - 1}\operatorname{sign}(X_s)\,  \di s$ exist,
  \item for any $t \in[0,\infty)$, \eqref{eq:main-simple} holds $\P$-a.s.
  \end{enumerate}
We say that $X$ spends zero time at $0$ if
\begin{equation}
\int_0^\infty \mathbb I(X_s=0)\,\di s=0 \quad \P\text{-a.s.}
\end{equation}
\end{definition}

Let us briefly summarize the
results. Denote by $\mathrm{BES}^\delta(z)$ a Bessel process of dimension $\delta$ starting at $z$.

It is well known that for $\delta\in[2,\infty)$ and $z\neq 0$, $\pm\mathrm{BES}^\delta(z)$ never hits zero, while
$\pm\mathrm{BES}^\delta(0)$ leaves zero immediately and never returns.

For $\delta\in(0,2)$, $\mathrm{BES}^\delta(z)$ visits $0$ infinitely often,
allowing the process to potentially change sign. In this regime, we will employ \emph{skew Bessel processes} to construct solutions.

Roughly speaking, a skew Bessel process with the skewness parameter $\theta\in[-1,1]$ is a strong Markov
process that behaves like a standard Bessel process on the positive half-line, like a negative Bessel process
on the negative half-line and a chooses the positive direction with $\frac{\theta\pm 1}{2}$
upon hitting zero.

A prominent example of this is the \emph{skew
Brownian motion} (see \cite{Lejay-06}), which corresponds to the skew Bessel process of dimension $\delta=1$.

For $\delta\in (-\infty ,0]$, $\mathrm{BES}^\delta(z)$
hits zero with probability one and remains there indefinitely.

Thus, the dimension parameter $\delta_{\alpha,\lambda}$ completely determines the dynamics of the solution at the origin.
Fig.~\ref{f:delta} illustrates the domains of different values of $\delta_{\alpha,\lambda}$
as a function of $\alpha$ and $\lambda$.

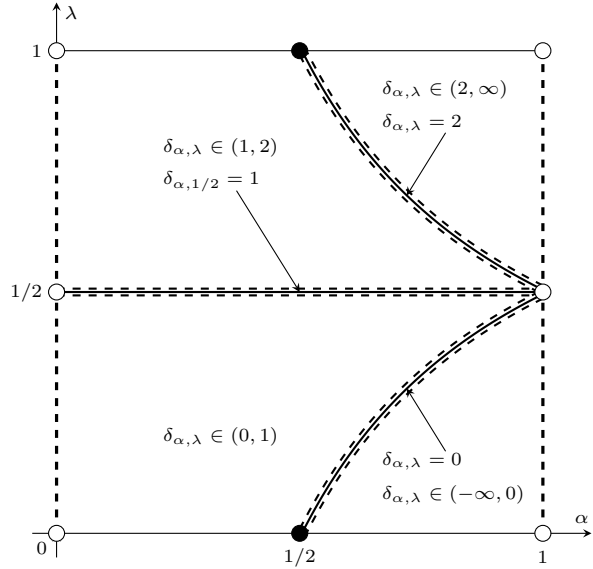
\begin{figure}[!h]
  \centering
  \begin{tikzpicture}
    \begin{axis}[
      width=0.6\textwidth,height=0.6\textwidth, hide axis,
      xmin=-0.1,
      xmax=1.2,
      ymin=-0.1,
      ymax=1.2,
      xlabel={$\alpha$},
      ylabel={$\lambda$},
      xtick = {0, 0.5, 1},
      xticklabels={$0$, $1/2$, $1$},
      ytick = {0, 0.5, 1},
      yticklabels={$0$, $1/2$, $1$},
      ytick distance=0.5,
      ylabel style={rotate=-90}
      ]
      \plot[name path=A, thick,samples=2,domain=0:1,dashed] {0.5+0.007};
      \plot[name path=A, thick,samples=2,domain=0:1,dashed] {0.5-0.007};


      \plot[name path=B, thick,samples=100,domain=0.5:1,dashed] {0.003+ 1 - 1 / (2*(x+0.003))};
      \plot[name path=B, thick,samples=100,domain=0.5:1] { -0.003+ 1 - 1 / (2*(x-.003))};
      \plot[name path=B, thick,samples=100,domain=0.51:1,dashed] {-0.009+ 1 - 1 / (2*(x-0.009))};

      \plot[name path=C, thick,samples=2,domain=0:1] {0.5};
      \plot[name path=B, thick,samples=100,domain=0.5:1,dashed] {-0.003 + 1 / (2*(x+0.003))};
      \plot[name path=B, thick,samples=100,domain=0.5:1] { 0.003+ 1 / (2*(x-.003))};
      \plot[name path=B, thick,samples=100,domain=0.51:1,dashed] {0.009 + 1 / (2*(x-0.009))};

      \plot[name path=E,  samples=2,domain=0:1] {1};
      \plot[very thick, dashed, samples=50,domain=0:1] coordinates {(0,0)(0,1)};
      \plot[very thick, dashed, samples=50,domain=0:1] coordinates {(1,0)(1,1)};

      \addplot[only marks,mark=*,mark options={scale=1.5, fill=white}] coordinates { (0,0)} ;
      \addplot[only marks,mark=*,mark options={scale=1.5, fill=white}] coordinates { (1,0)} ;
      \addplot[only marks,mark=*,mark options={scale=1.5, fill=black}] coordinates { (.5,0)} ;
      \addplot[only marks,mark=*,mark options={scale=1.5, fill=black}] coordinates { (.5,1)} ;
      \addplot[only marks,mark=*,mark options={scale=1.5, fill=white}] coordinates { (1,0.5)} ;
      \addplot[only marks,mark=*,mark options={scale=1.5, fill=white}] coordinates { (0,0.5)} ;
      \addplot[only marks,mark=*,mark options={scale=1.5, fill=white}] coordinates { (0,1)} ;
      \addplot[only marks,mark=*,mark options={scale=1.5, fill=white}] coordinates { (1,1)} ;


      \node[anchor=west] at (axis cs: 0.65,0.08) {$\delta_{\alpha,\lambda} \in(-\infty,0)$};
      \node[anchor=west] at (axis cs: 0.65,0.15) {$\delta_{\alpha,\lambda} =0$};

      \node[anchor=west ] at (axis cs: 0.2,0.2) {$\delta_{\alpha,\lambda}  \in (0, 1)$};
      \node[anchor=west ] at (axis cs: 0.2,0.8) {$\delta_{\alpha,\lambda}  \in (1, 2)$};
      \node[anchor=west ] at (axis cs: 0.2,0.73) {$\delta_{\alpha,1/2}  =1$};

      \node[anchor=west] at (axis cs: 0.65,0.92) {$\delta_{\alpha,\lambda} \in(2,\infty)$};
      \node[anchor=west] at (axis cs: 0.65,0.85) {$\delta_{\alpha,\lambda} =2$};

      \draw[->, >=stealth] (axis cs:-0.05,0) -- (axis cs:1.1,0);
      \draw[->, >=stealth] (axis cs:0,1) -- (axis cs:0,1.1);
      \draw[->, >=stealth] (axis cs:0,-0.05) -- (axis cs:0,0);

      \node at (axis cs: 1.08,.03) {$\alpha$};
      \node at (axis cs: 0.03,1.08) {$\lambda$};
      \node at (axis cs: 1,-0.05) {$1$};
      \node at (axis cs: 0.5,-0.05) {$1/2$};
      \node[anchor=east] at (axis cs: -0.0,-0.03) {$0$}; 
      \node[anchor=east] at (axis cs: -0.01,1) {$1$};     
      \node[anchor=east] at (axis cs: -0.01,0.5) {$1/2$}; 

      \draw[->, >=stealth] (axis cs:0.8,0.83) -- (axis cs:0.72,0.7);
      \draw[->, >=stealth] (axis cs:0.8,0.17) -- (axis cs:0.72,0.3);
      \draw[->, >=stealth] (axis cs:0.37,0.71) -- (axis cs:0.5,0.5);
    \end{axis}
  \end{tikzpicture}
\caption{Dimensions of the (skew) Bessel process as a function of the heterogeneity
index $\alpha\in(0,1)$ and the interpretation parameter $\lambda\in[0,1]$.\label{f:delta}}
\end{figure}

The main result is presented in the following Theorem.

\begin{theorem}
\label{thm:main_bessel}
Let $\alpha\in(0,1)$, $\lambda\in [0,1]$ and let
\ba
\delta:=\delta_{\alpha,\lambda}=\frac{1 - 2\alpha (1 - \lambda)}{1 - \alpha}\in\bR.
\ea
Furthermore, for $\theta\in[-1,1]$ assume that $Z^{\delta,\theta}(z)$ is a skew Bessel process
of dimension $\delta$ with skewness parameter $\theta$ started at $z\in\bR$.
Then for any $x\in\mathbb R$, the process
\ba
\label{e:X}
X_t = H^{-1}_\alpha(Z_t^{\delta,\theta}(H_\alpha(x))),\quad t\in [0,\infty),
\ea
is a weak solution of \eqref{eq:main-simple} starting at $x$.

For $\delta\in(0,\infty)$, this solution spends zero time at $0$,
For $\delta\in (-\infty,0]$, it
gets trapped at $0$ upon hitting it.

For $\delta\in (0,\infty)$, any weak solution of \eqref{eq:main-simple} starting at $x$,
that is a strong Markov process spending zero time at $0$,
has the law of the process $X$ defined in \eqref{e:X}.
For $\delta\in (-\infty,0]$, any weak solution of \eqref{eq:main-simple} starting at $x$
coincides in law with $X$.
 \end{theorem}

\begin{remark}
For $\delta \in[2,\infty)$ and $x \neq 0$, the solutions are strong; they are also strong if $z=0$ and
$\theta=\pm 1$.
For $\delta \in (1, 2)$ and $\theta\in\{-1,0,1\}$, the solutons are also strong, see \cite[Remark 2.30]{Blei12}.
For $\delta=1$, $X$ is a strong solution, see \cite{PavShe-20}.
For $\delta \in(-\infty,0]$, the solutions are strong, too.
\end{remark}

Our results are consistent with the theory of singular SDEs developed by \cite{CheEng05}.
In particular, the qualitative behavior of solutions to the SDE \eqref{eq:main-simple}, especially their behaviour at zero,
is described in Theorem 5.1 of Chapter 5 in that work.

The paper is organized as follows. In Section \ref{sec:bessel-processes} we introduce Bessel processes of
dimension $\delta\in (0,\infty)$
and provide their martingale characterization. Bessel processes of
dimension $\delta\in (-\infty,0]$ are introduced via the square Bessel process.
In Section \ref{s:skewBessel} we give a martingale characterization of Bessel processes
of dimension $\delta\in (0,2)$ and discuss trivial cases arising for other dimensions.
Section \ref{s:time} explains how a (skew) Bessel process can be obtained as a nonlinear
transformation of a time-changed Brownian motion. In Section \ref{sec:sdes}, we consider (skew)
Bessel processes as solutions of SDEs. Finally, Section \ref{s:proofs} contains
the proof of the main result, Theorem \ref{thm:main_bessel}.

\medskip

\noindent
\textbf{Acknowledgments:}
The authors thank the German Research Council (grant Nr.\ PA 2123/6-1) for financial support. The authors are grateful
to A.\ Pilipenko for stimulating discussions.

\section{Bessel processes\label{sec:bessel-processes}}

In this section, we review the fundamental properties of Bessel processes,
with a focus on their infinitesimal generators and martingale characterizations.

By $C_b(\bar\bR_+,\bR)$ we denote the space of real valued continuous bounded functions defined on $\bar\bR_+=[0,+\infty]$,
i.e., $f\in C_b(\bar\bR_+,\bR)$ is
bounded and continuous on $(0,+\infty)$ and the limits
\ba
f(0+)&:=\lim_{x\downarrow 0} f(x),\\
f(+\infty)&:=\lim_{x \uparrow +\infty} f(x),\\
\ea
exist and are finite.
Equipped with the supremum norm, the space $C_b(\bar\bR_+,\bR)$ is a Banach space.

A Bessel process of dimension $\delta\in(0,\infty)$
is the conservative diffusion process $Z^\delta$ with values in $\bR_+=[0,\infty)$ whose generator is
\ba
\label{eq:bessel_op}
L^\delta f(x) = \frac12 f''(x) + \frac{\delta - 1}{2x} f'(x)
\ea
on the domain
\ba
\label{e:Ddelta}
\rD_\delta=\{f\in C_b(\bar\bR_+,\bR)\cap C^2((0,\infty),\bR)\colon\
\lim_{x \downarrow 0} x^{\delta-1} f'(x)=0\ \text{and}\
L^\delta f\in C_b(\bar\bR_+,\bR)\}.
\ea
For the Bessel $Z^\delta$ starting at $z\in[0,\infty)$,
we will use notation $\mathrm{BES}^\delta(z)$.

The description of the domain of the operator $L^\delta$
follows from the general theory of Markov processes as presented in \cite[Chapter 16,\S 3]{Dynkin65} or
\cite[Chapter II]{mandl1968analytical}.

Indeed, following \cite{RevuzYor05}, p.\ 446, for $\delta\in (0,\infty)$, let us introduce a strictly increasing
continuous function $S_\delta\colon (0,\infty)\to\mathbb R$ (the scale function),
\ba
\label{e:Sdelta}
S(x)=S_\delta(x)=\begin{cases}
             x^{2-\delta},&\delta\in (0,2),\\
             2\ln x, & \delta=2,\\
            -x^{2-\delta}, & \delta\in (2,\infty),
            \end{cases}
\ea
and a positive continuous function $m_\delta \colon (0,\infty)\to \mathbb R_+$ (the speed measure density)
\ba
m_\delta(x) = \begin{cases} \frac{2}{2 - \delta} x^{\delta - 1}, &  \delta\in (0,2) ,\\
x, &  \delta = 2 \\
\frac{2}{\delta - 2} x^{\delta - 1}, &   \delta\in(2,\infty),
\end{cases}
\ea
and let
\ba
M(x)=M_\delta(x):=\int_0^x m_\delta(y)\,\di y.
\ea
Then for sufficiently smooth functions $f$, the operator \eqref{eq:bessel_op} can be written as
\ba
L^\delta f= \frac12 D_MD_S^+f ,
\ea
where $D_M$ and $D_S^+$ denote the (right-hand-side) derivatives with respect to the functions $M$ and $S$, i.e.\
\ba
D_S^+f(x)&=\lim_{y\downarrow x} \frac{f(y)-f(x)}{S_\delta(y)-S_\delta(x)}
=\begin{cases}
\frac{1}{2-\delta}x^{\delta-1}f'(x),\quad \delta\in (0,2),\\
\frac{x}{2}f'(x),\quad \delta=2,\\
\frac{1}{\delta-2}x^{\delta-1}f'(x),\quad \delta\in (2,\infty),\\
 \end{cases}\\
D_M f(x)&=\lim_{y\to x} \frac{f(y)-f(x)}{M_\delta(y)-M_\delta(x)}
=\begin{cases}
\frac{2-\delta}{2x^{\delta-1}}f'(x),\quad \delta\in (0,2) ,\\
\frac{1}{x}f'(x) ,\quad \delta=2,\\
\frac{\delta-2}{2x^{\delta-1}}f'(x),\quad \delta\in (2,\infty).
 \end{cases}
\ea
The domain of $D_MD_S^+f$ consists of functions
$f\in C_b(\bar\bR_+,\bR)$ such that
$f$ is continuously differentiable with respect to $S_\delta$ and
$D_S f$ is continuously differentiable with respect to $M$, and $D_MD_S^+f\in C_b(\bar\bR_+,\bR)$,
which leads to the definition \eqref{e:Ddelta}, see Section 4 in \cite{Blei12} for details.

Note that for $\delta\in[2,\infty)$, the condition $\lim_{x \downarrow 0} x^{\delta-1} f'(x)=0$ can be omitted.

Alternatively, we can write $L^\delta$ as
\ba
L^\delta f(x)=\frac12 \frac{1}{w_\delta(x)}\frac{\di}{\di x}(w_\delta(x)f'(x))
\ea
with
\ba
w_\delta(x)=x^{\delta-1}.
\ea
The operator $(L^\delta,\rD_\delta)$ defines the
a Feller transition function with the following density.

\begin{proposition}[A.2 in \cite{GJYor03}, p.\ 446 in \cite{RevuzYor05}]
  The transition probability of a Bessel process of dimension $\delta \in(0,\infty)$ is
  \ba
    \label{eq:trans-prob-bessel}
    p^\delta (t, x,y) &= t^{-1}  x^{-\nu} y^{\nu + 1} \ex^{-\frac{x^2 + y^2}{2t}} I_\nu(xy/t),\\
    p^\delta (t, 0, y) &= 2^{-\nu} t^{-\nu-1} \Gamma(\nu+1)^{-1} y^{2\nu + 1}\ex^{-\frac{y^2}{2t}},\quad \nu=\frac{\delta}{2}-1,
  \ea
where $I_\nu$ is the modified Bessel function of the first kind,
$I_\nu(z)=(z/2)^\nu\sum_{k=0}^\infty\frac{(z^2/4)^k}{k!\Gamma(\nu+k+1)}$.
\end{proposition}
The value $\nu=\frac{\delta}{2}-1$ is called the index of the Bessel process.
It is well known that for $\delta\in[2,\infty)$, $\{0\}$ is the instantaneous entrance point,
whereas for $\delta\in(0,2)$
it is an instantaneous reflecting point.

We now present a martingale characterization of the Bessel process, which will allow us to
identify a process as having the law of a Bessel process. For the proof, we adapt the approach of S.R.S.\ Varadhan,
who provided a similar characterization for reflected Brownian motion in \cite{Varadhan-RBM}.

By $C_b^2([0,+\infty),\bR)$ we understand a set of real valued functions defined on $[0,\infty)$ which are bounded
on $[0,\infty)$,
twice differentiable on $(0,\infty)$ and such that $f'$ and $f''$ are continuous and bounded on $(0,\infty)$.

\begin{lemma}
\label{p:DD}
Let
\ba
\widetilde \rD_\delta:=\{f\in C_b(\bar\bR_+,\bR)\cap C^2((0,+\infty),\bR)\colon
f'(0+)=0 \text{ and } L^\delta f \in C_b(\bar\bR_+,\bR)\}.
\ea
Then for any $\delta \in (0,\infty)$
\ba
\rD_\delta=\widetilde \rD_\delta.
\ea
Moreover, for $f\in \rD_\delta$ we have $f''(0+)=\frac{2L^\delta f(0+)}{\delta}$.
\end{lemma}
\begin{proof}
1. Let $f\in \rD_\delta$.
Let us show that $f'(0+)=0$ and $f''(0+)\in\bR$.
Let $g:= L^\delta f$. Then 
\ba
\label{e:11}
\lim_{x \downarrow 0} g(x) =  \frac12 \lim_{x \downarrow 0} \frac{1}{x^{\delta - 1}} \frac{\di }{\di x} (x^{\delta -1} f'(x)) = g(0+)\in\bR.
\ea
Hence, for every $\varepsilon > 0$, there exists $x_\e \in (0,\infty)$ 
such that 
\ba
\Big|\frac{1}{2x^{\delta - 1}} \frac{\di }{\di x} (x^{\delta - 1} f'(x)) - g(0+)\Big| < \varepsilon \quad \text{ for every }
0 < x \leq x_\e.
\ea
Hence,
\ba
2(g(0+) - \varepsilon) x^{\delta - 1} < \frac{\di }{\di x}(x^{\delta - 1}f'(x)) < 2(g(0+) + \varepsilon) x^{\delta - 1}.
\ea
  Since $\delta\in (0,\infty)$ we can integrate these inequalities on the interval $[a,x]$ for any
$0<a<x< x_\e$, and then pass to the limit as
  $a \downarrow 0$ to obtain
\ba
2(g(0+) - \varepsilon) \frac{x^\delta}{\delta}
\leq  x^{\delta - 1}f'(x) - \lim_{a\downarrow 0} a^{\delta-1} f'(a)
\leq 2(g(0+) + \varepsilon) \frac{x^\delta}{\delta}.
\ea
Since $\lim_{a  \downarrow 0} a^{\delta-1} f'(a) = 0$, we have
\ba
\label{e:15}
2(g(0+) - \varepsilon) \frac{x}{\delta} \leq f'(x) \leq  2(g(0+) + \varepsilon) \frac{x}{\delta}
\ea
resulting in both $f'(0+) = 0$ as $x\downarrow 0$ and
\ba
\label{e:16}
f''(0+)=\lim_{x  \downarrow 0 } \frac{f'(x)}{x} =\lim_{x  \downarrow 0 } \frac{f'(x) - f'(0+)}{x} = \frac{2g(0+)}{\delta}\in \bR.
\ea

\medskip

\noindent
2.\ Let $f\in \widetilde \rD_\delta$.
Now we show that
$\lim_{x \downarrow 0} x^{\delta-1} f'(x)=0$.
If $\delta=[1,\infty)$, then the statement is obvious.

Let $\delta\in(0,1)$.
Since $g(x)= L^\delta f(x)$ is continuous and bounded on $\bar\bR_+$,
it holds that for every $x \in (0, 1]$
\ba
\Big|\frac{\di }{\di x} (x^{\delta - 1}f'(x))\Big| \leq  2 \|g\|_\infty.
\ea
Hence, the function $x \mapsto x^{\delta - 1}f'(x)$ is continuous up to $x=0$.
Since $f'(0+)=0$, by the Lebesgue differentiation theorem we have that
\ba
\lim_{x \downarrow 0} x^{\delta - 1}f'(x) & = \lim_{a \downarrow 0} \frac{1}{a} \int_0^a y^{\delta - 1}f'(y)\,  \di y =0.
\ea
\end{proof}

Let us work on the canonical probability space of continuous real-valued functions equipped with the filtration $\mathbb F$
generated by the coordinate mappings.

\begin{theorem}[martingale characterization]
\label{t:bessel-mart-char}
Let $(Z_t)_{t \geq 0}$ be a one-dimensional continuous stochastic process.
Then $Z$ is a Bessel process of dimension $\delta\in (0,\infty)$ started at $z \in [0,\infty)$ if and only if
\begin{enumerate}
\item[i)]
$Z_0 = z$ a.s.,
\item[ii)]
for every $f \in C_c^2((0, \infty),\bR)$, the process
\begin{equation}
M^f_t = f(Z_t) - f(z) - \int_0^t L^\delta f(Z_s)\,\di s
\end{equation}
is a martingale,
\item[iii)]
the process $Z$ spends zero time at zero, i.e.,
\ba
\int_0^\infty \mathbb I_{\{0\}}(Z_s) \diff{s} = 0\text{ a.s.},
\ea
\end{enumerate}

\end{theorem}

\begin{proof}
1. First, let us assume that $Z$ is a Bessel process of dimension $\delta \in(0,\infty)$, started at $z\in[0,\infty)$.
Then \emph{i)} and \emph{iii)} hold true.
For any $f\in C^2_c((0,\infty),\bR) \subset \rD_q=\widetilde \rD_\delta$, Dynkin's formula (see, e.g., Lemma 17.21
in \cite{Kallenberg-02}) yields \emph{ii)}.

\medskip
\noindent
2. Assume that a continuous process $Z$ satisfies \emph{i)}, \emph{ii)} and \emph{iii)}, and let $\P_z$
be its probability measure.

Assume first that \emph{ii)} holds for any $f\in\widetilde \rD_\delta$.
Then
\ba
\mathbf E_z f(Z_t) =  f(z) + \int_0^t \mathbf E_z L^\delta (Z_s)\, \di s
\ea
and hence for each $\lambda>0$
\ba
\lambda \int_0^\infty \ex^{-\lambda t}\mathbf E_z f(Z_t)\,\di t
=  f(z) + \int_0^\infty \ex^{-\lambda t }\mathbf E_z L^\delta (Z_t)\, \di t
\ea
or equivalently
\ba
\label{e:res}
&\int_0^\infty \ex^{-\lambda t}\mathbf E_z \Big(\lambda f(Z_t) -   L^\delta f(Z_t)\Big)  \,\di t =  f(z).
\ea
Let $R_\lambda$ denote the resolvent operator of $Z$, namely, let
\ba
R_\lambda g(z):= \int_0^\infty \ex^{-\lambda t}\mathbf E_z g(Z_t)\,\di t,
\ea
and let $R_\lambda^\delta$ be the resolvent of the Bessel semigroup \eqref{eq:trans-prob-bessel},
\ba
R_\lambda^\delta g (z): =\int_0^\infty \ex^{-\lambda t} \int_0^\infty g(y)p^\delta(t,z,y)\,\di y\,\di t.
\ea
Since the Bessel semigroup is a Feller semigroup,
$R^\delta_\lambda g \in\rD_\delta$ for any $g\in C_b(\bar\bR_+,\bR)$ and for any
$g\in C_b(\bar\bR_+,\bR)$ there is a unique $f\in\rD_\delta$ be such that
\ba
\lambda f-L^\delta f=g.
\ea
Combining this with \eqref{e:res} we get that for any $g\in C_b(\bar\bR_+,\bR)$
\ba
R_\lambda g  =  f= R_\lambda^\delta g
\ea
which implies that 
\ba
\label{e:EZ}
\mathbf E_z f(Z_t)= \int_0^\infty f(y) p^\delta(t,z,y)\,\di y
\ea
for almost all $t\in[0,\infty)$. Since $Z$ is continuous, \eqref{e:EZ} holds for all $t\in[0,\infty)$, so that
$Z$ has the same one-dimensional distributions as a Bessel process. By the
uniqueness Theorem 4.10.1 in \cite{Kolokoltsov-11},
$Z$ is a strong Markov process, and hence, is a Bessel process of dimension $\delta\in (0,\infty)$.

\marginpar{To finish, approximations}

To complete the proof we need to show that \emph{ii)} holds for every $f \in \widetilde \rD_\delta= \rD_\delta$.
Without loss of generality we can assume that $f\in \widetilde\rD_\delta$ satisfies $f(0)=0$.

For $f\in \widetilde\rD_\delta$, we construct an approximating sequence
$\{f_n\}_{n\in\mathbb N}\subset C^2_c(\bR_+,\bR)$ as follows. First, we set
\ba
\widetilde f''_n(x)&=2n\Big(x-\frac{1}{2n}\Big)f''(0+)\mathbb I_{[\frac{1}{2n},\frac{1}{n})}(x)
+f''\Big(x-\frac1n\Big)\mathbb I_{[\frac{1}{n},\infty)}(x),\\
\widetilde f'_n(x)&=n(x-\frac{1}{2n})^2f''(0+)\mathbb I_{[\frac{1}{2n},\frac{1}{n})}(x) +
\Big(\frac1{4n}f''(0+)+ f'(x-\frac1n)\Big)\mathbb I_{[\frac{1}{n},\infty)}(x),\\
\widetilde f_n(x)&=\frac{n}{3}(x-\frac{1}{2n})^3 f''(0+)\mathbb I_{[\frac{1}{2n},\frac{1}{n})}(x) +
\Big( \frac1{24 n^2} f''(0+) + \frac{x-1/n}{4n}f''(0+)+ f(x-\frac1n)\Big) \mathbb I_{[\frac{1}{n},\infty)}(x).
\ea
Let $\eta\in C^\infty(\bR,\bR)$ be such that $\eta(x)=1$ for $x\in(-\infty,0]$ and $\eta(x)=0$ for $x\in [1,\infty)$.
Then we set
\ba
\eta_n(x)&=\eta(x-n),\\
f_n(x) &= \widetilde f_n(x)\eta_n(x),
\ea
so that $f_n\in C^2_c(\bR_+,\bR)$. Furthermore, we have pointwise convergence
\ba
&\lim_{n\to\infty} f_n(z)=f(z),\quad z\in [0,\infty).
\ea
The straightforward calculation yields:
\ba
L^\delta f_n(x)=\begin{cases}
\displaystyle
                 0, & x\in [0,1/2n)\cup [n+1,\infty),\\
\displaystyle
f''(0+)\Big[n\Big(x-\frac{1}{2n}\Big) + \frac{\delta-1}{2x} n\Big(x-\frac{1}{2n}\Big)^2\Big], & x\in [1/2n,1/n),\\
\displaystyle
L^\delta f(x-1/n) +\frac{\delta-1}{2x}\frac1{4n}f''(0+),& x\in [1/n,n),\\
\displaystyle
L^\delta \Big[f(x -1/n)\eta_n(x)\Big]
+ \Big( \frac1{24 n^2} f''(0+) + \frac{x-1/n}{4n}f''(0+) \Big)L^\delta \eta_n(x), & x\in [n,n+1),\\
                \end{cases}
\ea
and it is clear that
\ba
\sup_{n\in\mathbb N}\|L^\delta \eta_n\|_\infty<\infty,\quad n\in\mathbb N.
\ea
Furthermore,
\ba
L^\delta\Big[f(x -1/n)\eta_n(x)\Big]
&=\frac12\Big( f''(x -1/n)\eta_n(x)  +2 f(x -1/n)\eta'_n(x) + f(x -1/n)\Big)\\
&+ \frac{\delta-1}{2x } \Big( f'(x -1/n)\eta_n(x)  +f(x -1/n)\eta'_n(x) \Big)\\
&= \eta_n(x)L f(x -1/n) + f(x -1/n)\eta'_n(x) + \frac12 f(x -1/n)
+  \frac{\delta-1}{2x }f(x -1/n)\eta'_n(x).
\ea
Combining these estimates we get that
\ba
&\lim_{n\to\infty} L^\delta f_n(z)=L^\delta f(z),\quad z\in (0,\infty).
\ea
and therefore
\ba
\sup_n\|L^\delta f_n\|_\infty<\infty.
\ea
By \emph{ii)}, for each $n\in\mathbb N$ the process $M^{f_n}$ is a martingale.

Let us consider the limit of $M_t^{f_n}$ as $n \to \infty$. First, due to the convergence described
above, we have $f_n(Z_t) \to f(Z_t)$ for each $t\in [0,\infty)$.
Then, for each $\e\in(0,1]$
\ba
\label{e:ww}
|M_t-M_t^{f_n}|
&\leq |f_n(Z_t) - f(Z_t)| + |f_n(z)- f(z)| \\
&+ \int_0^t |L^\delta f(Z_s) -L^\delta f_n(Z_s)|\mathbb{I}_{[0, \e)}(Z_s)\,\di s\\
&+
\int_0^t |L^\delta f(Z_s) -L^\delta f_n(Z_s)|\mathbb{I}_{[\e, \infty)}(Z_s)\,\di s.
\ea
Since the sequence $\{Lf_n\}_{n\in\mathbb N}$ is uniformly bounded in the supremum norm, for each $\e\in (0,1]$  and  $t\in(0,\infty)$
by \emph{iii)}
we have
by Fatou's Lemma that
\ba
\limsup_{\e \downarrow 0}\int_0^t  |L^\delta f(Z_s) -L^\delta f_n(Z_s)|\mathbb{I}_{[0, \e)}(Z_s)\, \di s
&\leq C \limsup_{\e \downarrow 0}\int_0^t \mathbb I_{[0, \e)}(Z_s)\, \di s\\
&\leq C \int_0^t \limsup_{\e \downarrow 0}\mathbb I_{[0, \e)}(Z_s)\, \di s=0\quad \text{a.s.}
\ea
The last integral in \eqref{e:ww} converges to $0$ for any $\e\in (0,1]$ as $n\to\infty$, too.
Finally, we note that $|M_t-M_t^{f_n}|\leq C$ a.s., and the
the bounded convergence theorem implies that $M^f$ is a martingale.
\end{proof}

A Bessel process of dimension $\delta\in\bR$ can also be realized with the help of the so-called Square Bessel Process. Let $B$ be a standard Brownian motion.
Then, for any $y\in[0,\infty)$ and $\delta\in \bR$, the SDE
\ba
Y_t=y +\delta t +2 \int_0^t \sqrt{Y_s}\,\di B_s
\ea
has a unique strong solution, which we denote $\mathrm{BESQ}^\delta(y)$.
For $\delta\in(0,\infty)$, $Y$ is a non-negative strong Markov process, and $\sqrt{Y}=\text{BES}^\delta(\sqrt y)$,
see Definition 1.9 in \cite{RevuzYor05}.

For $\delta\in(-\infty,0]$, and $y\in (0,\infty)$, the process $Y$ reaches $0$ in finite time and stays there.
Hence, for $\delta\in(-\infty,0]$, the Bessel process can be defined only until the first hitting time $\tau_0$.

For $\delta\in (-\infty,2)$, the transition density of the Bessel process killed at $0$
equals
\ba
\label{e:bessel-killed}
p^{\delta,\dag} (t, x , y)
=
\begin{cases}
t^{-1}  x^{-\nu} y^{\nu + 1} \ex^{-\frac{x^2 + y^2}{2t}} I_{-\nu}(xy/t), &  \delta \in [0,2),
\ \nu=\delta/2-1,\\
t^{-1}  x^{\widetilde \nu} y^{-\tilde \nu + 1} \ex^{-\frac{x^2 + y^2}{2t}} I_{\tilde\nu}(xy/t),
& \delta \in (-\infty,0),\ \tilde \nu=|\delta|/2+1.
\end{cases}
\ea
Note, that $I_{-1}=I_1$.

\section{Skew Bessel processes\label{s:skewBessel}}

In the previous section, we studied Bessel processes taking values on the positive half-line $\bR_+$.

A natural question is whether a Bessel process can be extended to the negative half-line in
such a way that the resulting process is strong Markov, spends zero time at $0$, and behaves like
$\pm\mathrm{BES}$
away from zero. This can be achieved using \emph{skew Bessel processes}.

Since a Bessel process cannot reach the origin from a starting point $z\in(0,\infty)$ for $\delta\in[2,\infty)$,
or is gets trapped at the origin for $\delta\in(-\infty,0]$,
a non-trivial skew Bessel process can only be constructed for dimensions $\delta\in(0,2)$.

For completeness, we first consider the trivial cases.

Let $\delta\in [2,\infty)$ and $\theta\in [-1,1]$.
We define the skew Bessel process $Z^{\delta,\theta}(z)$ as follows:
\ba
&\text{for }z\in (0,\infty)\quad Z^{\delta,\theta}(z)=\mathrm{BES}^\delta(z),\\
&\text{for }z\in (-\infty,0)\quad Z^{\delta,\theta}(z)=-\mathrm{BES}^\delta(|z|).
\ea
For $z=0$, we randomize the sign of the process and set
\ba
\mathrm{Law} (Z^{\delta,\theta}(0))=\frac{1+\theta}{2} \mathrm{Law} (Z^\delta(0)) +\frac{1-\theta}{2} \mathrm{Law} (-Z^\delta(0)).
\ea

Let $\delta\in(-\infty,0]$. Then for any $\theta\in [-1,1]$ and any $z\in (0,\infty)$ we set $Z^{\delta,\theta}(z)=Z^{\delta}(z)$
for any $z\in (-\infty,0)$ we set $Z^{\delta,\theta}(z)=-Z^{\delta}(|z|)$, and
$Z^{\delta,\theta}(0)\equiv 0$.
These processes equal to zero for $t\in[\tau_0,\infty)$.

In the nontrivial case $\delta\in (0,2)$, let us first consider the skewness parameter $\theta\in (-1,1)$.

We define the strictly increasing scale function
\ba
\label{eq:skew-scale}
S_{\delta,\theta}(x)
=\frac{2}{\operatorname{sign}(x)+ \theta} |x|^{2 - \delta}
=
\begin{cases}
\frac{2}{1+\theta} x^{2 - \delta}, & x \in[0,\infty), \\
-\frac{2}{1-\theta} |x|^{2 - \delta}, & x\in(-\infty, 0),
\end{cases}
\ea
and the speed measure $M_{\delta,\theta}$ with the
density
\ba
\label{e:mdt}
m_{\delta,\theta}(x)=
 \frac{1-\theta}{2-\delta}|x|^{\delta-1}\mathbb I_{(-\infty,0)}(x)
+\frac{1+\theta}{2-\delta}x^{\delta-1}\mathbb I_{(0,\infty)}(x).
\ea
We say that
$C_b(\bar\bR,\bR)$ is a space of real valued bounded functions that are continuous on $\bR$ and such that the limits
\ba
f(+\infty)&:=\lim_{x \uparrow +\infty} f(x),\quad f(-\infty):=\lim_{x \downarrow -\infty} f(x)
\ea
exist and are finite.
Consider the operator
\ba
L^{\delta,\theta}f(x)=\frac12 D_{M_{\delta,\theta}}D^+_{S_{\delta,\theta}}f(x)
\ea
on the domain
\ba
\rD_{\delta,\theta}
=\Big\{
f&\in C_b(\bar\bR,\bR)\cap C^2(\mathbb R\backslash \{0\},\bR)\colon
(1-\theta) \lim_{x\uparrow 0}|x|^{\delta-1}f'(x)= (1+\theta) \lim_{x\downarrow 0} x^{\delta-1}f'(x),\
L^{\delta } f\in C_b(\bar\bR,\bR),
\Big\}
\ea
It is clear that on the domain $\rD_{\delta,\theta}$, the operator $L^{\delta,\theta}$ takes the form
\eqref{eq:bessel_op}.

Then, the operator $(L^{\delta},\rD_{\delta,\theta})$ generates a
strongly continuous (Feller) semigroup. The corresponding process is the
skew Bessel process.

For $\theta=1$, we define the skew Bessel process as follows:
\ba
Z^{\delta,1}(z)=Z^{\delta}(z),\quad \text{for } z\in[0,\infty)
\ea
and for $z\in (-\infty,0)$
\ba
Z^{\delta,1}_t(z) =\begin{cases}
- Z^{\delta}(|z|),\quad t\in[0,\tau_0),\\
Z^{\delta}_{t-\tau_0}(0),\quad t\in[\tau_0,\infty).
                   \end{cases}
\ea
For $\theta=-1$, the construction is analogous.

It is clear, that $|Z^{\delta,\theta}|$ is a $\delta$-dimensonal Bessel process for any $\theta\in[-1,1]$.

\begin{proposition}[\cite{alili2019semi}, Theorem 2]
  The transition probability of a skew Bessel process of dimension $\delta\in (0,2)$ is
\ba
\label{eq:trans-prob-skew-bessel}
p^{\delta,\theta}(t, z, x)  =
\begin{cases}
    \displaystyle
\Big[\frac{1 + \theta}{2} I_\nu\Big(\frac{|xy|}{t}\Big) + \frac{1 - \theta}{2} I_{-\nu} \Big( \frac{|xy|}{t}\Big)\Big]
 t^{-1}|x|^{-\nu}|y|^{\nu+1} \ex^{-\frac{x^2 + y^2}{2t}}& x > 0 \text{ and } y > 0 \\
\displaystyle
\Big[\frac{1 - \theta}{2}I_\nu\Big(\frac{|xy|}{t}\Big)  - \frac{1 - \theta}{2} I_{-\nu}\Big(\frac{|xy|}{t}\Big) \Big]
 t^{-1}|x|^{-\nu}|y|^{\nu+1} \ex^{-\frac{x^2 + y^2}{2t}}
& x > 0 \text{ and } y < 0 \\
\displaystyle
 \Big[\frac{1+\theta}{2} I_\nu\Big(\frac{|xy|}{t}\Big) -   \frac{1+\theta}{2}  I_{-\nu}\Big(\frac{|xy|}{t}\Big) \Big]
  t^{-1}|x|^{-\nu}|y|^{\nu+1} \ex^{-\frac{x^2 + y^2}{2t}}
 & x < 0 \text{ and } y > 0 \\
 \displaystyle
\Big[\frac{1-\theta}{2} I_\nu\Big(\frac{|xy|}{t}\Big)  + \frac{1+\theta}{2} I_{-\nu}\Big(\frac{|xy|}{t}\Big) \Big]
 t^{-1}|x|^{-\nu}|y|^{\nu+1} \ex^{-\frac{x^2 + y^2}{2t}}
& x < 0 \text{ and } y < 0
    \end{cases}
\ea
  with
\ba
F(t, 0, x) = \lim_{z \to 0} F(t, z, x).
\ea
\end{proposition}

The transition density \eqref{eq:trans-prob-skew-bessel} can be written in the unified form as
\ba
p^{\delta,\theta}(t,x,y)=p^{\delta,\dag}(t,|x|,|y|)\mathbb I_{(0,\infty)}(xy)
+ \frac{1+\theta\operatorname{sign} y}{2}\Big(p^\delta(t,|x|,|y|) - p^{\delta,\dag}(t,|x|,|y|) \Big)
\ea
where $p^{\delta,\dag}$ is the transition density of a $\delta$-dimensional Bessel process killed at $0$
given in \eqref{e:bessel-killed}

We have the following martingale characterization of the skew Bessel processes, see Definition 3 in \cite{watanabebilateral}.

\begin{theorem}[semimartingale characterization]
\label{t:sbessel-mart-char}
A continuous Markov stochastic process $(Z_t)_{t\in[0,\infty)}$ defined on a filtered space
$(\Omega, \rF, (\rF_t), \P)$ is equal in law to a skew Bessel process of dimension
$\delta \in (0, 2)$ started at $z \in \mathbb{R}$ with skewness parameter $\theta \in [-1, 1]$ if and only if
\begin{enumerate}[label=(\roman*)]
\item
$Z_0 = z$ a.s.,

\item
For every $f \in C_c^2(\mathbb{R} \backslash \{0\})$, the process
\ba
      M^f_t = f(Z_t) - f(z) - \int_0^t L^{\delta,\theta} f(Z_s) \, \di s
\ea
is a martingale;
\item
the process $Z$ spends zero time at zero, i.e.,
$\displaystyle \int_0^t \mathbb I_{\{0\}}(Z_s)\,  \di s = 0$ a.s.,
\item
the function $S_{\delta,\theta}$ defined in \eqref{eq:skew-scale}
is a scale function of $Z$, i.e., $(S_{\delta,\theta}(Z_t))_{t\in[0,\infty)}$ is a local martingale.
\end{enumerate}
\end{theorem}
\begin{proof}
1.
Let assume first that $Z$ is a skew Bessel process of dimension $\delta\in (0,2)$
started at $z\in\bR$ with skewness parameter $\theta$.
Then \textit{(i)} and \textit{(iii)} is obvious, \textit{(iv)} follows from the Proposition 3.5 Chapter VII in \cite{RevuzYor05}.
The statement \textit{(ii)} holds due to standard results in semigroup theory, see Chapter VII,
Proposition 1.6 in \cite{RevuzYor05}.

2.
We proceed as in Theorem \ref{t:bessel-mart-char} 2., and first assume that $\theta\in (-1,1)$.
We have to show that for every $f \in D_{\delta,\theta}$ the process $M^f_t$ is a martingale.

Let $f \in D_{\delta,\theta}$. Without loss of generality we assume that $f(0)=0$. There is a finite limit
\ba
\lim_{x \to 0} \frac{f'(x)}{S'_{\delta,\theta}(x)}
=\lim_{x \uparrow 0} \frac{(1-\theta)f'(x)}{2(2-\delta)|x|^{1-\delta}}
=\lim_{x \downarrow 0} \frac{(1+\theta)f'(x)}{2(2-\delta)x^{1-\delta}}
=: a\in\bR.
\ea
Let $h(x): = f(x) - a S_{\delta,\theta}(x)$. Then we have that $L^{\delta,\theta} f(x) = L^{\delta,\theta} h(x)$, $x\in\bR$, $h(0)=0$, and
\ba
\lim_{x \downarrow 0} x^{\delta-1} h'(x)
&=\lim_{x \downarrow 0} x^{\delta-1} \Big( f'(x) - a \frac{2(2-\delta)}{1+\theta} x^{1-\delta} \Big)
=\lim_{x\downarrow 0} \frac{f'(x)}{x^{1-\delta}} -  \frac{2(2-\delta)a}{1+\theta}  =0,\\
\lim_{x\uparrow 0} |x|^{\delta-1} h'(x)
&= \lim_{x\uparrow 0} |x|^{\delta-1} \Big(f'(x) - a \frac{2(2-\delta)}{1-\theta}   |x|^{1-\delta}\Big)
=\lim_{x\uparrow 0} \frac{f'(x)}{|x|^{1-\delta}} -\frac{2(2-\delta) a }{1-\theta}  =0.
\ea
Literally repeating the steps \eqref{e:11}--\eqref{e:15} for the limits as $x\downarrow 0$ and  $x\uparrow 0$, we
get that $h'(0+)=h'(0-)=0$. Furthermore, as in \eqref{e:16} we get that
\ba
h''(0+)=h''(0-)=\frac{L^{\delta,\theta} f(0)}{\delta}.
\ea
Now, as in the argument of Theorem~\ref{t:bessel-mart-char}, we approximate $h$ by a sequence $\{h_n\}\subset C^2_c(\bR\backslash\{0\})$
and conclude that $M^h$ is a local martingale. Eventually, we note that
\ba
M^f_t= M^h_t + a S_{\delta,\theta}(Z_t)-aS_{\delta,\theta}(z)
\ea
is a local martingale, too. Since $(M^f_t)_{t\in[0,T]}$ is bounded for each $T\in[0,\infty)$, we get that $M^f$ is a martingale.

The one-sided cases $\theta=\pm 1$ are obtained as in Theorem~\ref{t:bessel-mart-char}.
\end{proof}

\section{Bessel and skew Bessel process as a time changed Brownian motion, $\delta\in(0,2)$\label{s:time}}

The following representations of the Bessel and skew Bessel process as a time changed Brownian motion
will be used in the sequel for construction of weak solutions of the heterogeneous diffusion equation \eqref{eq:main-simple}.

We start with Bessel processes. Let $\delta\in(0,2)$, and recall that $S_\delta(x)=x^{2-\delta}$,
$x\in[0,\infty)$, given in \eqref{e:Sdelta}
is a scale function of a Bessel process $Z^\delta$. Let
\ba
R_\delta(x)=x^{\frac{1}{2-\delta}},\quad x\in[0,\infty),
\ea
be its inverse function.
\begin{theorem}
\label{t:equiv-bessel}
Let $\delta \in (0, 2)$, $z \in [0,\infty)$ and let $\beta=(\beta_t)_{t\in[0,\infty)}$
be a standard Brownian motion started at $S_{\delta}(z)$. Let $r=(r_t)_{t\in[0,\infty)}$ be a random time change defined as
\ba
t = \int_0^{r_t} (R'_{\delta}(|\beta_s|))^2\,\di s.
\ea
Then, the process $Z^{\delta}:= (R_{\delta}(|\beta_{r_t}|))_{t\geq 0}$
is a Bessel process of dimension $\delta$.
\end{theorem}
\begin{proof}
For $\delta \in (0, 2)$, $R'_{\delta}(x)=\mathcal O (|x|^{\frac{1}{2-\delta}-1})$ as $|x|\to 0$,
with $2(\frac{1}{2-\delta}-1)=2 \frac{\delta-1}{2-\delta}\in (-1,+\infty)$. Hence
$(R'_{\delta})^2\in L^1_{\text{loc}}(\bR_+,\bR)$ and the random time change
\ba
\tau_t = \int_0^{t} (R'_{\delta}(|\beta_s|))^2\,\di s
\ea
is a.s.\ finite for all $t\in[0,\infty)$.
Since $\beta$ spends zero time at 0, the random time change $t\mapsto \tau_t$ is
strictly increasing and invertible, so that $t\mapsto r_t$ is well defined and is strictly increasing.

We use the martingale characterization in Theorem \ref{t:bessel-mart-char}. Clearly, $Z^\delta$ has continuous paths and
starts at $z$.

Let $f \in C_c^2((0,\infty),\bR)$. Then $g(\cdot):=f\circ R_{\delta}\circ |\,\cdot\,|\in C_c^2(\mathbb{R} \backslash \{0\})$, and
the process
\ba
N^g_t = g(\beta_t) - g(\beta_0) - \frac12 \int_0^t g''(\beta_s)\,\di s
\ea
is a martingale w.r.t.\ to the filtration $(\rF_t)_{t\in[0,\infty)}$, so that the time changed process $(N^g_{r_t})_{t\in[0,\infty)}$ is a
martingale w.r.t.\ to the filtration $(\rF_{r_t})_{t\in[0,\infty)}$.

A simple calculation yields:
\ba
M^f_t&= f( R_{\delta}(|\beta_{r_t}|)) - f(z)
- \frac12 \int_0^{r_t}
f''( R_{\delta}(|\beta_s|)) (R'_{\delta}(|\beta_s|))^2
+ f'( R_{\delta}(|\beta_s|)) R''_{\delta}(|\beta_s|) \,\di s\\
&= f( R_{\delta}(|\beta_{r_t}|)) - f(z)
- \frac12 \int_0^t
f''( R_{\delta}(|\beta_{r_s}|))
+ f'( R_{\delta}(|\beta_{r_s}|)) R''_{\delta}(|\beta_{r_s}|) (R'_{\delta}(|\beta_{r_s}|))^{-2} \,\di s.
\ea
Taking into account that
\ba
R''_{\delta}(x) (R'_{\delta}(x))^{-2}=\frac{\delta-1}{R_{\delta}(x)},\quad x\in (0,\infty),
\ea
we get that the process
\ba
M^f_t &= f( Z^{\delta}_t) - f(z)
- \int_0^t
L^\delta f( Z^{\delta}_s)  \,\di s
\ea
is a $(\rF_{r_t})$-martingale.

Clearly, $Z^\delta$ spends zero time at zero, because $(\beta_{r_t})_{t\in[0,\infty)}$ does so.
\end{proof}

To get the similar result for the skew Bessel process, we introduce the inverse of $S_{\delta,\theta}$ defined in \eqref{eq:skew-scale}:
\ba
R_{\delta,\theta}(z)&=S_{\delta,\theta}^{-1}(z)
=
-\Big(\frac{1-\theta}{2}\Big)^{\frac{1}{2-\delta}}|z|^{\frac{1}{2-\delta}}\mathbb I_{(-\infty,0)}(z)
+\Big(\frac{1+\theta}{2}\Big)^{\frac{1}{2-\delta}} z ^{\frac{1}{2-\delta}}\mathbb I_{( 0,\infty)}(z).
\ea

\begin{theorem}
\label{t:equiv-skew-bessel}
Let $\delta \in (0, 2)$, $\theta \in (-1, 1)$, $z \in \mathbb{R}$ and let $\beta=(\beta_t)_{t\in[0,\infty)}$
be a standard Brownian motion started at $S_{\delta,\theta}(z)$. Let $r=(r_t)_{t\in[0,\infty)}$ be a random time change defined as
\ba
\label{e:rskew}
t = \int_0^{r_t} (R'_{\delta,\theta}(\beta_s))^2\,\di s.
\ea
Then, the process $Z^{\delta,\theta}:= (R_{\delta,\theta}(\beta_{r_t}))_{t\in[0,\infty)}$
is a skew Bessel process of dimension $\delta$ and with skewness parameter $\theta$ started at $z$.
\end{theorem}

\begin{proof}
The proof of this statement is repeats the proof of Theorem \ref{t:equiv-bessel}.
We check the conditions of the martingale characterization Theorem \ref{t:sbessel-mart-char}. Clearly, $Z$ has continuous paths and
starts at $z$. If $f \in C_c^2(\mathbb{R} \backslash \{0\})$, then
$g:=f\circ R_{\delta,\theta}\in C_c^2(\mathbb{R} \backslash \{0\})$, and
the process
\ba
N^g_t = g(\beta_t) - g(\beta_0) - \frac12 \int_0^t g''(\beta_s)\,\di s
\ea
is an $(\rF_t)$-martingale, and $(N^g_{r_t})_{t\in[0,\infty)}$ is an
$(\rF_{r_t})$-martingale, as well as the process $M^f$.
The process $Z^{\delta,\theta}$ spends zero time at $0$. Finally, the process
$S_{\delta,\theta}(Z^{\delta,\theta})=(\beta_{r_t})_{t\in[0,\infty)}$ is an $(\rF_{r_t})$-martingale.
\end{proof}

\section{Bessel and skew Bessel processes as solutions of SDEs\label{sec:sdes}}

The Bessel and skew Bessel process in Section \ref{sec:bessel-processes} turn out to be solutions of the SDEs
with unbounded or singular drift.

This connection is well know for the Bessel process, see \cite[Chapter XI]{RevuzYor05},
but it is a bit more subtle for the skew Bessel processes.

Two notions of a local time will be used in this section: the semi-martingale local times and a local time
with respect to a speed measure $m(\di a)$.

For a continuous semimartingale $X=(X_t)_{t\in[0,\infty)}$, the field of (symmetric) local times $L(X)=(L^a_t(X))_{a\in\bR,t\in[0,\infty)}$
is the a.s.\ unique random field for which the occupation times formula
\ba
\int_0^t \varphi(X_s)\, \di \langle X \rangle_s = \int_{\mathbb{R}} \varphi(a) L^a_t(X)\, \di a
\ea
holds true a.s.\ for each measurable non-negative function $\varphi$, see
\cite[Chapter IV]{RevuzYor05}. The symmetric semimartingale local times can be calculated as the a.s.\ limit
\ba
L^a_t(X)=\lim_{\e\downarrow 0}\frac{1}{2\e}\int_0^t \mathbb I_{[a-\e,a+\e]}(X_s)\,\di \langle X\rangle_s.
\ea
The mapping $a\mapsto L_t^a(X)$ is a.s.\ right-continuous. It is a.s.\ continuous is $X$ is a martingale.

On the other hand, a real valued Markov diffusion $X=(X_t)_{t\in[0,\infty)}$ with a speed measure $m(\di a)$ given,
a random field $\ell(X)=(\ell^a_t(X))_{a \in\bR,t\in[0,\infty)}$ is said to be a local
time with respect to $m$ if the following holds a.s.:
\ba
\int_0^t \varphi(X_s)\, \di s = \int_{\mathbb{R}} \varphi(a) \ell^a_t(X)\, m(\di a).
\ea
A local time $\ell(X)$ is unique up to null sets of the measure $m$.

The following results for the Bessel processes are generally known.
\begin{theorem}
\label{thm:bessel_to_sdes}
Let $Z=Z^\delta$ be a Bessel process of dimension $\delta\in(0,\infty)$ started at $z\in[0,\infty)$, and let $B$
be a standard Brownian motion.
\begin{enumerate}
\item
If $\delta \in(1,\infty)$ then $Z$ is the unique strong solution of the SDE
\ba
\label{eq:bsde}
Z_t = z + B_t + \int_0^t \frac{\delta-1}{2Z_s}\, \di s.
\ea
  \item
    If $\delta = 1$, then $Z$ is the unique strong solution of the SDE
    \begin{equation}
      \label{eq:bsde_reflected}
      Z_t = z + B_t + L^0_t(Z).
    \end{equation}
  \item
    If $\delta \in (0, 1)$, then $Z$ is the unique strong solution of the SDE
\ba
\label{eq:bsde_pv}
Z_t &= z + B_t +  \text{\emph{p.v.}} \int_0^t \frac{\delta - 1}{2Z_s}\, \di s\\
&= z + B_t +  \int_0^\infty \frac{\delta - 1}{2}(\ell^a_t(Z) - \ell^0_t(Z)) a^{2-\delta} \, \di a\\
&= z + B_t +  \int_0^\infty \frac{\delta - 1}{2a}  (\ell^a_t(Z) - \ell^0_t(Z))\, m_\delta(\di a),
\ea
where $\ell^Z_t(a)$ is a local time of $Z$ with respect to the speed measure $m_\delta(\di a)=\frac{2}{2-\delta}a^{\delta - 1}\, \di a$.
  \end{enumerate}
\end{theorem}

\begin{proof}
The proofs of statements \emph{1.} and \emph{2.} can be found in \S 1 of Chapter XI in \cite{RevuzYor05}
and Theorem 3.2 in \cite{cherny2000strong},
and Section 3 in \cite{HShepp-81}, respectively.

\emph{3}. The existence of a weak solution
to the SDE \eqref{eq:bsde_pv}
can be found in \cite{RevuzYor05} in Exercise 1.26 of Chaper XI.
The existence and uniqueness of the strong solution follows from Theorem 1
and Example 2 of \cite{AryasovaP-11}.

For completeness, we give here the proof of the weak existence.
Let $\beta$ be a standard Brownian motion started at $S_\delta(z)$ defined in \eqref{e:Sdelta}.
The process $Z=(R_\delta(|\beta_{r_t}|))_{t\in[0,\infty)}$ is a Bessel process of dimension $\delta$
by Theorem \ref{t:equiv-bessel}, where $R_\delta(x)=x^{\frac{1}{2-\delta}}$, $x\in[0,\infty)$.
In the sequel we omit the subscript $\delta$ and denote $S:=S_\delta$ and $R:=R_\delta$.

Let $L^a(\beta)$ be the symmetric local time of the Brownian motion $\beta$ at $a\in\bR$.
By Eq.\ (1.17) in Chapter VI (p.~232) in \cite{RevuzYor05}, for $a\in (0,\infty)$ and $t\in[0,\infty)$
\ba
\label{e:La}
L^a(|\beta|)_t=L^a_t(\beta)  + L^{-a}_t(\beta)
\ea
and
\ba
\label{e:L0}
L^0_t(|\beta|)=2L^0_t(\beta).
\ea
The following transformation formula holds for $a\in(0,\infty)$:
\ba
L^{R(a)}_t(Z)  = R'(a) L^a_{r_t}(|\beta|),
\ea
i.e.,
\ba
L^{a^{1/(2-\delta)}}_t(Z)  = \frac{a^{(\delta-1)/(2-\delta)}}{2-\delta} L^a_{r_t}(|\beta|)
\ea
see (1.23) in Chapter VI in \cite{RevuzYor05}. Denoting $z:=a^{1/(2-\delta)}$, we see that
the following limit holds a.s.\
\ba
\lim_{z \downarrow 0} z^{1 - \delta} L^z_t(Z) = \frac{2}{2-\delta}L^0_{r_t}(\beta).
\ea
Let us define the random field
\ba
\ell_t^a(Z):=a^{1-\delta} L^Z_t(a)= \frac{1}{2-\delta}L^{S(a)}_{r_t}(|\beta|) ,
\ea
so that
\ba
L^{a}_{r_t}(|\beta|)&= (2-\delta)\ell_t^{R(a)}(Z),\\
\ell_t^0(Z)&=\frac{1}{2-\delta}L^{0}_{r_t}(|\beta|).
\ea
For $\e\in (0,1)$, consider a family of symmetric functions $g_\varepsilon$ that approximate
$g(x):=R(|x|) = |x|^{\frac{1}{2 - \delta}}$.
We define
\ba
g'_\e(x):=R'(\e\vee x),\quad x\in[0,\infty),
\ea
and set
\ba
g_\e(x)&:=\int_0^x g'_\e(y)\,\di y, \quad x\in[0,\infty),\\
g_\e(x)&:= g_\e(-x), \quad x\in(-\infty,0).
\ea
The function $g_\e$ has a derivative which is a function of bounded variation.
Applying the It\^o--Tanaka--Meyer formula to $g_\varepsilon(\beta_t)$ we obtain
\ba
g_\varepsilon(\beta_t)& = g_\varepsilon(\beta_0)
+ \int_0^t g'_\e(\beta_s)\,  \di \beta_s
+ \frac12 \int_{\mathbb{R}} L^a_t (\beta)\, g_\e''(\di a)\\
&=g_\varepsilon(\beta_0)
+ \int_0^t g'_\e(\beta_s)\,  \di \beta_s
+ \frac12 \int_{\mathbb{R}} L^a_t (\beta)\mathbb I(|a|>\e) g_\e''(a)\,\di a+ \frac12 (g'(\e) - g'(-\e)) L^0_t(\beta).
\ea
We rewrite the difference as follows:
\ba
g'(\e) - g'(-\e)=-\int_{\mathbb R} g''_\e(a)\mathbb I(|a|>\e)\,\di a,
\ea
and get the equality
\ba
\label{e:gg}
g_\varepsilon(\beta_t) = g_\varepsilon(\beta_0)
+ \int_0^t g'_\e(\beta_s)\,  \di \beta_s
+ \frac12 \int_{\mathbb{R}} g''(a)\mathbb I(|a|>\e) (L^a_t (\beta) -L^0_t (\beta))  \,  \di a.
\ea
Applying the time change $t\mapsto r_t$ to \eqref{e:gg} and adding and subtracting $\int_0^{r_t} g'(\beta_s)\,  \di \beta_s$ yields
\ba
g_\e(\beta_{r_t}) = g_\varepsilon(\beta_0)
& + \int_0^{r_t} g'(\beta_s)\,  \di \beta_s + \int_0^{r_t} (g'_\e(\beta_s)- g'(\beta_s))\,  \di \beta_s\\
&+ \frac12 \int_{\mathbb{R}} g''(a)\mathbb I(|a|>\e) (L^a_{r_t} (\beta) -L^0_{r_t} (\beta))  \,  \di a.
\ea
Therefore taking into account \eqref{e:La} and \eqref{e:L0} we get
\ba
\int_{\mathbb{R}} g''(a)\mathbb I(|a|>\e) (L^a_t (\beta) -L^0_t (\beta))  \,  \di a
&=\int_\e^\infty g''(a)\mathbb (L^a_t (\beta) + L^{-a}_t(\beta)-2L^0_t (\beta))  \,  \di a\\
&=\int_\e^\infty g''(a)\mathbb (L^a_t (|\beta|)-L^0_t (|\beta|))  \,  \di a.
\ea
Passing to the limit as $\e\downarrow 0$ yields
\ba
\int_0^\infty g''(a)  \Big(L^a_{r_t} (|\beta|)-L^0_{r_t} (|\beta|)\Big)  \,  \di a
&=\int_0^\infty (2-\delta) R''(a)\Big( \ell_t^{R (a)}(Z) -\ell_t^0 (Z)\Big)  \,  \di a\\
&=\int_0^\infty R''(S(a)) S'(a) \Big( \ell_t^a(Z) -\ell_t^0 (Z)\Big)  \,  \di a\\
&=(\delta-1)\int_0^\infty a^{\delta-2} \Big( \ell_t^a(Z) -\ell_t^0 (Z)\Big)  \,  \di a.
\ea
Since $g'\in L^2_{\text{loc}}(\bR,\bR)$ we have convergence
\ba
\int_0^{r_t} g'_\e(\beta_s)\,  \di \beta_s\to \int_0^{r_t} g'(\beta_s)\,  \di \beta_s
\ea
in probability.
By the time change formula \cite[Theorem 8.5.7]{oksendal2003stochastic} the process $B=(B_t)_{t\in[0,\infty)}$ given by
\begin{equation}
B_t:=\int_0^{r_t} g'(\beta_s)\, \di \beta_s
\end{equation}
is a Brownian motion adapted to $(\rF_{r_t})$, and we get the formula \eqref{eq:bsde_pv}.
\end{proof}

Now we formulate similar results for the skew Bessel process with $\delta\in (0,2)$ and $\theta\in(-1,1)$.

\begin{theorem}\label{thm:skew_bessel_to_sdes}
  Let $Z=Z^{\delta,\theta}$ be a skew Bessel process of dimension $\delta \in (0,2)$
  with skewness parameter $\theta \in (-1, 1)$.
\begin{enumerate}
  \item
If $\delta \in  (1, 2)$, then $Z$ is the unique weak solution of the SDE
\ba
Z_t &= z + B_t + \int_0^t \frac{\delta - 1}{2Z_s}\, \di s
\ea
whose symmetric semimartingale local times satisfy the balance condition
\ba
\label{eq:blei}
& (1 - \theta) \lim_{a \downarrow 0} a^{1 - \delta} L^a_t(Z) = (1 + \theta) \lim_{a \uparrow 0} |a|^{1 - \delta} L^a_t(Z)\quad a.s.
\ea

\item
If $\delta = 1$, then $Z$ is the unique strong solution of the SDE
    \begin{equation}
      \label{eq:bsde_skew_bm}
      Z_t = z + B_t + \theta L^0_t(Z),
    \end{equation}
where $L^0(Z)$ is the symmetric semimartingale local time of $Z$ at $0$.
  \item
If $\delta \in (0, 1)$, then $Z$ is a weak solution of the SDE
\ba
\label{eq:bsde-skew-pv}
Z_t & = z + B_t + \text{\emph{p.v.}} \int_0^t \frac{\delta - 1}{2 Z_s} \, \di s\\
&= z + B_t
+ \int_{\mathbb R}  \frac{\delta-1}{2a} \Big(\ell_t^a(Z) - \ell_t^0(Z)\Big)\,  m_{\delta,\theta}(\di a).
\ea
where $\ell(Z)$ is the family of local times w.r.t.\ the measure $m_{\delta,\theta}(\di a)$ defined in \eqref{e:mdt}.

  \end{enumerate}
\end{theorem}
\begin{proof}
The statement \emph{1.} follows from Theorem 2.22 in \cite{Blei12} (existence and uniqueness of the SDE) and Section 4 in \cite{Blei12}
(identification of a solution as a skew Bessel process).

For \emph{2.} see Section 3 in \cite{HShepp-81}.

The proof of \emph{3.} goes along the lines of the proof of \emph{3.}
in Theorem \ref{thm:bessel_to_sdes}. Let $\beta$ be a standard Brownian motion
started at $S_{\delta,\theta}(z)$. Then, by Theorem \ref{t:equiv-skew-bessel}
the process $Z_t=R_{\delta,\theta}(\beta_{r_t})$ is a skew Bessel process.

For $z\neq 0$, the local times of $\beta$ and $Z$ are related as
\ba
L^{R(a)}_t(Z)  = R'(a) L^a_{r_t}(\beta)
\ea
(see (1.23) Chapter VI in \cite{RevuzYor05})
or, equivalently,
\ba
L^a_t(Z)  = R'(S(a)) L^{S(a)}_{r_t}(\beta)
\ea
Taking into account that
\ba
R'(S(a)) =
-\frac{1}{2-\delta}\frac{1-\theta}{2}|a|^{\delta-1}\mathbb I_{(-\infty,0)}(z)
+\frac{1}{2-\delta}\frac{1+\theta}{2}a^{\delta-1}\mathbb I_{[0,\infty)}(z)
\ea
and that the local time of $\beta$ is continuous, we see that the following limit holds a.s.\
\ba
\frac{2}{1-\theta}\lim_{a \uparrow 0} |a|^{1 - \delta} L^a_t(Z) =
\frac{2}{1+\theta}\lim_{a \downarrow 0} a^{1 - \delta} L^a_t(Z) =
\frac{1}{2-\delta}L^0_{r_t}(\beta).
\ea
For $\e\in (0,1)$, consider a family of functions $g_\varepsilon$ that approximate $g(x):=R(x)$.
We define
\ba
g'_\e(x):=\begin{cases}
   R '(\max\{\e, x\}),\quad x\in[0,\infty),\\
   R '(\min\{-\e, x\}),\quad x\in (-\infty,0),
          \end{cases}
\ea
and set
\ba
g_\e(x)&:=\int_0^x g'_\e(y)\,\di y, \quad x\in\bR.
\ea
The function $g_\e$ has a derivative which is a function of bounded variation.
Applying the It\^o--Tanaka--Meyer formula to $g_\varepsilon(\beta_t)$ we obtain
\ba
g_\varepsilon(\beta_t)& = g_\varepsilon(\beta_0)
+ \int_0^t g'_\e(\beta_s)\,  \di \beta_s
+ \frac12 \int_{\mathbb{R}} L^a_t (\beta)\, g_\e''(\di a)\\
&=g_\varepsilon(\beta_0)
+ \int_0^t g'_\e(\beta_s)\,  \di \beta_s
+ \frac12 \int_{\mathbb{R}} L^a_t (\beta)\mathbb I(|a|>\e) g''(a)\,\di a+ \frac12 (g'(\e) - g'(-\e)) L^0_t(\beta).
\ea
Rewriting
\ba
g'(\e) - g'(-\e)=-\int_{\mathbb R} g''_\e(a)\mathbb I(|a|>\e)\,\di a
\ea
we get
\ba
g_\varepsilon(\beta_t) = g_\varepsilon(\beta_0)
+ \int_0^t g'_\e(\beta_s)\,  \di \beta_s
+ \frac12 \int_{\mathbb{R}} g''(a)\mathbb I(|a|>\e) (L^a_t (\beta) -L^0_t (\beta))  \,  \di a.
\ea
We apply the time change $r_t$
and define the random field
\ba
\ell_t^a(Z):=\begin{cases}
   \displaystyle           \frac{2}{1+\theta} a^{1-\delta} L^a_t(Z)= \frac{1}{2-\delta}L^{S(a)}_{r_t}(\beta) , &a\in (0,\infty),\\
    \displaystyle           \frac{2}{1-\theta} |a|^{1-\delta} L^a_t(Z)= \frac{1}{2-\delta}L^{S(a)}_{r_t}(\beta) , &a\in (-\infty,0),\\
    \displaystyle           \frac{1}{2-\delta}L^{0}_{r_t}(\beta) , &a=0.
             \end{cases}
\ea
Then we have
\ba
\int_{\mathbb R}  g''(a)  \Big(L^a_{r_t} (\beta)-L^0_{r_t} (\beta)\Big)  \,  \di a
&=(2-\delta)\int_{\mathbb R}   R'' (a)\Big( \ell_t^{R (a)}(Z) -\ell_t^0(Z)\Big)  \,  \di a\\
&=(2-\delta)\int_{\mathbb R}
R'' (S (a)) S' (a) \Big( \ell_t^a(Z) -\ell_t^0 (Z)\Big)  \,  \di a\\
&=\frac{\delta-1}{2}\int_{0}^\infty \frac{1+\theta}{2-\delta}a^{\delta-2} \Big( \ell_t^a(Z) -\ell_t^0 (Z)\Big)  \,  \di a\\
&-\frac{\delta-1}{2}\int_{-\infty}^0 \frac{1-\theta}{2-\delta}|a|^{\delta-2} \Big( \ell_t^{a}(Z) -\ell_t^0 (Z)\Big)  \,  \di a\\
&=\frac{\delta-1}{2}\int_{\mathbb R} \frac{1}{a} \Big( \ell_t^a(Z) -\ell_t^0 (Z)\Big) m(a) \,  \di a,
\ea
with $m(a)$ defined in \eqref{e:mdt}.
Since $g'\in L^2_{\text{loc}}(\bR,\bR)$ we have
\ba
\int_0^{r_t} g'_\e(\beta_s)\,  \di \beta_s\to \int_0^{r_t} g'(\beta_s)\,  \di \beta_s
\ea
in probability.
By the time change formula \cite[Theorem 8.5.7]{oksendal2003stochastic} the process $B=(B_t)_{t\geq 0}$ given by
\begin{equation}
B_t:=\int_0^{r_t} g'(\beta_s)\, \di \beta_s
\end{equation}
is a Brownian motion adapted to the filtration $(\rF_{r_t})$.
\end{proof}

\section{Proof of Theorem \ref{thm:main_bessel}\label{s:proofs}}

\begin{lemma}
\label{lemma:skew-bessel-to-1}
Let $\alpha\in(0,1)$ and $\lambda\in [0,1]$ be such that
$\delta=\delta_{\alpha, \lambda} \in (0, 2)$. Let $x\in\mathbb R$ and $\theta\in(-1,1)$ and
let $Z^{\delta,\theta}$ be a skew Bessel process started at $z=H_\alpha(x) \in \mathbb{R}$.
Then, the process
$X = (H^{-1}_\alpha(Z^{\delta,\theta}_t(z)))_{t\in[0,\infty)}$ is a weak solution of \eqref{eq:main-simple}
started at $x$. The semimartingale local time of $X$ satisfies the following balance equation at zero:
\ba
\label{e:ltX}
(1 - \theta) \lim_{a \downarrow 0}  a^{-2\alpha\lambda} L^a_t(X) = (1 + \theta) \lim_{a \uparrow 0} |a|^{-2\alpha\lambda} L^a_t(X).
\ea
\end{lemma}
\begin{proof}
Let $\beta=(\beta_t)_{t\in[0,\infty)}$ be a standard Brownian motion started at $\beta_0=S_{\delta,\theta}(z)$.
By Theorem \ref{t:equiv-skew-bessel},  the process $Z^{\delta,\theta}:= (R_{\delta,\theta}(\beta_{r_t}))_{t\in[0,\infty)}$
is a skew Bessel process of dimension $\delta$ and skewness $\theta$, where $r$ is the random time change
defined in \eqref{e:rskew}.
Consider the function
\ba
f(x) := H^{-1}_\alpha(R_{\delta,\theta}(x))
=
\Big[\Big(\frac{1-\alpha}{(1-\theta)^{\frac{1}{2-\delta}}}\Big)^{\frac{1}{1-\alpha}}\mathbb I_{[0,\infty)}(x)
- \Big(\frac{1-\alpha}{(1+\theta)^{\frac{1}{2-\delta}}}\Big)^{\frac{1}{1-\alpha}}\mathbb I_{(-\infty,0)}(x)\Big]
|x|^{\frac{1}{(2-\delta)(1-\alpha)}}.
\ea
Since
\ba
  \frac{1}{(2-\delta)(1-\alpha)}= \frac{1}{1 - 2 \alpha \lambda}\in(1,\infty),
\ea
the derivative $f'$ is absolutely continuous, $f' \in \operatorname{AC}(\mathbb R,\mathbb{R})$,
so that Krylov's generalized It\^o formula
yields
\ba
f(\beta_t) = f(\beta_0) + \int_0^t f'(\beta_s)\, \di \beta_s  + \frac12 \int_0^t f''(\beta_s)\, \di s.
\ea
By the time change formula \cite[Theorem 8.5.7]{oksendal2003stochastic}, the process
\ba
\label{e:B1}
B_t : = \int_0^{r_t} R'_{\delta,\theta}(\beta_s)\,  \di \beta_s,\quad t\in[0,\infty),
\ea
is a Brownian motion adapted to the filtration $(\rF_{r_t})_{t\in[0,\infty)}$.
Therefore by the time change formula we get the equality
\ba
\label{e:BB}
\int_0^{r_t} f'(\beta_s)\, \di \beta_s
&=\int_0^t f'(\beta_{r_s}) r'_s \, \di \beta_{r_s}\\
&= \int_0^t f'(\beta_{r_s}) (R'_{\delta,\theta}(\beta_{r_s}))^{-2}  \, \di B_s\\
&= \int_0^t |f(\beta_{r_s})|^\alpha \, \di B_s.
\ea
The change of variables in the Lebesgue integral yields
\ba
\label{e:LL}
\frac12 \int_0^{r_t} f''(\beta_s)\, \di s &= \frac12 \int_0^t f''(\beta_{r_s}) r'_s \, \di s\\
&=\frac12 \int_0^t f''(\beta_{r_s}) (R'_{\delta,\theta}(\beta_{r_s}))^{-2} \, \di s\\
&=\alpha\lambda \int_0^{t} |f(\beta_{r_s})|^{2\alpha-1} \operatorname{sign}{\beta_{r_s}}\, \di s.
\ea
In other words, the processes $X_t=f(\beta_{r_t})=H^{-1}_\alpha(Z^{\delta,\theta}_t)$, $t\in[0,\infty)$, and $B$
defined in
\eqref{e:B1} solve the SDE
\ba
X_t = x + \int_0^t |X_s|^\alpha\, \di B_s  + \alpha\lambda \int_0^t |X_s|^{2 \alpha -1}\operatorname{sign}(X_s) \,\di s.
\ea
The semimartingale local time of $X$ is obtained from the local time of $\beta$:
\ba
L^X_{t}(x) = \frac{H'_\alpha(x)}{S_{\delta,\theta}'(H_\alpha(x))} L^\beta_{r_t}(S_{\delta,\theta}(H_\alpha(x))).
\ea
Observing that the local time $x\mapsto L^\beta_t(x)$ of $\beta$ is continuous, we get the balance equation
\eqref{e:ltX}
for $L^X$ at zero.
\end{proof}

Now we consider the cases $\theta=\pm1$. By symmetry, it is sufficient to consider the case $\theta=1$.

\begin{lemma}
\label{l:bessel-to-1}
Let $\alpha\in(0,2)$ and $\lambda\in [0,1]$ be such that
$\delta=\delta_{\alpha, \lambda} \in (0,\infty)$. Let $x\in[0,\infty)$ and
let $Z^{\delta}$ be a Bessel process started at $z=H_\alpha(x) \in [0,\infty)$.
Then, the process $X = (H^{-1}_\alpha(Z^{\delta}_t(z)))_{t \geq 0}$ is a non-negative weak solution of \eqref{eq:main-simple}
started at $x$.
\end{lemma}
\begin{proof}
1. Let $\delta \in [2,\infty)$ and $Z^\delta$ be a Bessel process of dimension $\delta$ that satisfies the SDE
\eqref{eq:bsde}.
Note that in this case $\alpha\in [1/2,1)$, so that $H_\alpha^{-1} \in C^2([0,\infty),\bR)$.
Applying the It\^o formula
to $H_\alpha^{-1}(Z^\delta)$ yields the result.

\noindent
2. For $\delta \in (0, 2)$ we recall Theorem \ref{t:equiv-bessel}. Let $\beta$ be a Brownian motion started at $S_\delta(z)$. Then
$Z^\delta=R_\delta(|\beta_{r_t}|)$ is a Bessel process of dimension $\delta$.

Let $f(x) = H^{-1}_\alpha(R_\delta(|x|))$
and consider the process $X = (f(\beta_{r(t)}))_{t \geq 0}$.
Since $f' \in \operatorname{AC}([0, \infty))$ we conclude that
\ba
X_t = x + \int_0^{r(t)} f'(\beta_s)\, \di \beta_s + \frac12 \int_0^{r(t)} f''(\beta_s)\,  \di s
\ea
by the Krylov generalized It\^o formula.
The process
\ba
B_t = \int_0^{r(t)} R'_\delta(|\beta_s|)\, \di \beta_s
\ea
is a Brownian motion
adapted to the filtration $(\rF_{r_t})$. Repeating the formulae \eqref{e:BB} and \eqref{e:LL} yields the result.
\end{proof}

The last two lemmas allow constructing weak solutions of \eqref{eq:main-simple}
by transforming a Bessel or a skew Bessel process with the help of a non-linear mapping $H_\alpha^{-1}$.
We now show that all solutions of \eqref{eq:main-simple}
are of that form. We first deal with just non-negative solutions and then extend it
to all weak solutions after the identification of the scale function.

\begin{lemma}
\label{lemma:from-positive-1-to-bessel}
Let $\delta = \delta_{\alpha, \lambda}\in (0,\infty)$,
and let $(X, B)$ be a weak solution of \eqref{eq:main-simple} started at $x \in[0,\infty)$
such that $X$ is non-negative, time homogeneous strong Markov process spending zero time at $0$.
Then, the process $Z = (H_\alpha(X_t))_{t \geq 0}$ is a Bessel process of dimension $\delta$ started at $z = H_\alpha(x)$.
\end{lemma}
\begin{proof}
Let $f \in C^2_c((0, \infty),\bR)$, and set 
$h(x) := f(H_\alpha(x))$, so that $h \in C^2_c((0,\infty),\bR)$.
Moreover, $h$ satisfies for $x\in[0,\infty)$:
  \begin{align*}
    h'(x) &= f'(H_\alpha(x)) x^{-\alpha},\\
    h''(x) &= f''(H_\alpha(x)) x^{-2\alpha} - \alpha f'(H_\alpha(x)) x^{-\alpha - 1}.
  \end{align*}
The It\^o formula yields
\ba
\label{eq:4}
f(Z_t) = h(X_t)&
=h(x) + \int_0^t h'(X_s) X_s^\alpha\, \di B_s
+ \frac12 \int_0^t \Big( X_s^{2\alpha} h''(X_s) + 2 \lambda \alpha X_s^{2\alpha - 1} h'(X_s) \Big)\, \di s\\
&=f(z) + \int_0^t f'(Z_s) X_s^{-\alpha}  X_s^\alpha\,  \di B_s \\
&+ \frac12 \int_0^t  X_s^{2\alpha} \Big(f''(Z_s) X_s^{-2\alpha} - \alpha f'(Z_s) X_s^{-\alpha - 1} \Big)\, \di s
+\lambda \int_0^t \alpha X_s^{2\alpha - 1} f'(Z_s) X_s^{-\alpha}\, \di s\\
&=f(z) + \int_0^t f'(Z_s)\, \di B_s
+ \frac12 \int_0^t f''(Z_s) + (2 \lambda  - 1) \alpha  X_s^{\alpha - 1} f'(Z_s)\,  \di s\\
&=f(z) + \int_0^t f'(Z_s) \, \di B_s
+ \frac12 \int_0^t f''(Z_s) + \frac{(2 \lambda  - 1) \alpha}{(1 - \alpha)H(X_s)} f'(Z_s)\,  \di s\\
&=f(z) + \int_0^t f'(Z_s) \, \di B_s
+ \int_0^t\Big( \frac12 f''(Z_s) + \frac{(2 \lambda  - 1) \alpha}{(1 - \alpha)} \frac{f'(Z_s)}{2 Z_s}\Big)\,  \di s\\
&=f(z) + \int_0^t f'(Z_s) \, \di B_s + \int_0^t \Big(\frac12 f''(Z_s) + (\delta - 1) \frac{f'(Z_s)}{2 Z_s}\Big)\,  \di s.
\ea
Therefore the process 
\ba
M^f_t = f(Z_t) - f(z) - \int_0^t \Big(\frac12 f''(Z_s) + \frac{\delta - 1}{2Z_s} f'(Z_s)\Big)\,\di s
\ea
is a continuous martingale. Since the process $Z$ spends zero time at $0$,
the result follows from Theorem \ref{t:bessel-mart-char}.
\end{proof}

In order to obtain a similar result for two-sided weak solutions of \eqref{eq:main-simple} we have to
identify the scale function of the process $H_\alpha(X)$.

\begin{lemma}
\label{prop:from-bessel-to-skew-bessel}
Let $Y$ be a continuous time homogeneous strong Markov process started at $y \in \mathbb{R}$ 
such that $|Y|$ is a Bessel process of dimension $\delta \in (0, 2)$ started at $|y|$. Then, there is $\theta \in [-1, 1]$ 
such that $S_{\delta,\theta}$ is a scale function of $Y$, i.e., $S_{\delta,\theta}(Y)$ is a local martingale.
\end{lemma}
\begin{proof}
Following the proof of \cite[Proposition 7.1]{PavShe-20},
for $a<b$
we define the first exit time
\ba
\tau_{(a, b)} = \inf\set{t \in[0,\infty) \colon Y_t \not\in (a, b)}.
\ea
First we show that the probability 
\ba
p_+(\varepsilon) = \P_0( Y_{\tau_{(-\varepsilon, \varepsilon)}} = \varepsilon)
\ea
does not depend on $\varepsilon$.
Let $0 < \varepsilon < \varepsilon'$, then
\ba
p_+(\e')
&= \P_\e(Y_{\tau_{(-\e', \e')}} = \e' ) p_+(\e)
+ \P_{-\e}(Y_{\tau_{(-\e', \e')}} = \e' ) (1 - p_+(\e)) .
\ea
Since $|Y|$ is a Bessel process of dimension $\delta \in (0,2)$, 
the function $S_{\delta}(x) = x^{2 - \delta}$, $x\in[0,\infty)$, is its scale function. Therefore
\ba
\P_{-\e}(Y_{\tau_{(-\e',0)}} = -\e') &= \P_\e(Y_{\tau_{(0, \e')}} = \e')
= \frac{S_\delta(\e) - S_\delta(0)}{S_\delta(\varepsilon') - S_\delta(0)}
= \Big(\frac{\e}{\e'}\Big)^{2 - \delta} ,\\
\P_{-\e}(Y_{\tau_{(-\e', 0)}} = 0)
&= \P_\e(Y_{\tau_{(0, \e')}} = 0 ) = \frac{S_\delta(\e') - S_\delta(\e)}{S_\delta(\e') - S_\delta(0)}
= 1 - \Big(\frac{\e}{\e'}\Big)^{2 - \delta} .
\ea
  Hence,
\ba
\P_\e(Y_{\tau_{(-\e', \e')}} = \e')
&= \P_\e(Y_{\tau_{(0, \e')}} = \e') + \P_\e(Y_{\tau_{(0, \e')}} = 0) p_+(\e') \\
&= \Big(\frac{\e}{\e'}\Big)^{2 - \delta} + \Big(1 - \Big(\frac{\e}{\e'}\Big)^{2 - \delta}\Big) p_+(\e'),
\ea
  and
\ba
\P_{-\e}(Y_{\tau_{(-\e', \e')}} = \e'  )
= \P_{-\e}(Y_{\tau_{(-\e', 0)}} = 0 ) p_+(\e')
= \Big(1 - \Big(\frac{\e}{\e'}\Big)^{2 - \delta}\Big) p_+(\e').
\ea
Finally, we get that
\ba
p_+(\e') = \Big[\Big(\frac{\e}{\e'}\Big)^{2 - \delta} + \Big(1 - \Big(\frac{\e}{\e'}\Big)^{2 - \delta}\Big) p_+(\e')\Big] p_+(\e)
+ \Big[\Big(1 - \Big(\frac{\e}{\e'}\Big)^{2 - \delta}\Big) p_+(\e')\Big] (1 - p_+(\e)),
\ea
so that
\ba
p_+(\e') = p_+(\e)=:p_+\in[0,1].
\ea
Let $\theta =2 p_+ - 1$
and $a < 0 < b$. Without loss of generality, let $|a|<b$. First assume that $x = 0$. Since $Y$ is a strong Markov process, we get
\ba
\P_0(\tau_b < \tau_a)
&=\P_0(Y_{\tau(a,b)}=b) = p_+ + (1-p_+)\P_{-b}(Y_{\tau(a,0)}=0)\P_0(Y_{\tau(a,b)}=b)\\
&=\P_0(Y_{\tau(a,b)}=b) = p_+
+ (1-p_+) \frac{S_\delta(|a|) - S_\delta(b)}{S_\delta(|a|) - S_\delta(0)} \P_0(Y_{\tau(a,b)}=b),
\ea
so that
\ba
\P_0(Y_{\tau(a,b)}=b)
=\frac{S_\delta(|a|) p_+}{S_\delta(b) p_- + S_\delta(|a|) p_+}
= \frac{S_{\delta,\theta}(0)-S_{\delta,\theta}(a)}{S_{\delta,\theta}(b) - S_{\delta,\theta}(a)}
\ea
with $S_{\delta,\theta}$ defined in \eqref{eq:skew-scale}.
For $x \in (0, b)$ we get
\ba
\P_x(\tau_b < \tau_a)= \P_x(Y_{\tau(a,b)}=b)
&=\P_x(Y_{\tau(b,0)}=b) + (1-  \P_x(Y_{\tau(b,0)}=b)    )  \P_0(Y_{\tau(a,b)}=b)\\
&=
\frac{S_\delta(x)- S_\delta(0) }{S_\delta(b)- S_\delta(0)}
+ \frac{S_\delta(b)- S_\delta(x) }{S_\delta(b)- S_\delta(0)}
\P_0(Y_{\tau(a,b)}=b),
\ea
so that
\ba
\P_x(\tau_b < \tau_a)
&=\frac{S_{\delta,\theta}(x)-S_{\delta,\theta}(a)}{S_{\delta,\theta}(b) - S_{\delta,\theta}(a)}.
\ea
Analogously,
for $x \in (a, 0)$
\ba
\P_x(\tau_b < \tau_a)  &=
\frac{S_{\delta,\theta}(x)-S_{\delta,\theta}(a)}{S_{\delta,\theta}(b) - S_{\delta,\theta}(a)}.
\ea
It is clear, that the same equalities hold true for $a=-b$,
$0<a<x<b$ and $a<x<b<0$.
Therefore, $S_{\delta,\theta}$ is a scale function of $Y$.
\end{proof}

We apply the previous result to identify solutions of \eqref{eq:main-simple} taking values on $\mathbb{R}$.

\begin{lemma}
Let $\alpha\in (0,1)$ and $\lambda\in[0,1]$ be such that $\delta=\delta_{\alpha,\lambda}\in (0,2)$, and let
$(X, B)$ be a weak solution of \eqref{eq:main-simple} started at $x \in \mathbb{R}$ such that
$X$ is a time-homogeneous strong Markov process spending zero time at $0$.
Let $Z = (H_\alpha(X_t))_{t \geq 0}$ started at $z = H_\alpha(x)$.
Then there is $\theta \in [-1, 1]$ such that $Z$ is a skew Bessel process of dimension $\delta$ started at
$z = H(x)$ with skewness parameter $\theta$.
\end{lemma}
\begin{proof}
Arguing as in Lemma \ref{lemma:from-positive-1-to-bessel} we conclude that $|Z|$ is indeed a Bessel process of dimension
$\delta_{\alpha, \lambda}\in (0,2)$ started at $|z|$.
Lemma
\ref{prop:from-bessel-to-skew-bessel} implies that there is
$\theta \in [-1, 1]$ such that
$S_{\delta,\theta}$ is the scale function of $Z$.
Let $f \in C^2_c(\mathbb{R} \backslash \{0\})$, and define $h(x) := f(H_\alpha(x))$, so that
$h \in C^2_c(\mathbb{R} \backslash \{0\})$ as well. Eventually we repeat the argument of Lemma \ref{lemma:from-positive-1-to-bessel}
and use the fact that $S_{\delta,\theta}(Z)$ is a local martingale to identify $Z$ as a skew Bessel process
via the martingale characterization in Theorem \ref{t:sbessel-mart-char}.
\end{proof}

\begin{lemma}
Let $\alpha\in (0,1)$ and $\lambda\in[0,1]$ be such that $\delta=\delta_{\alpha,\lambda}\in [2,\infty)$, and let
$(X, B)$ be a weak solution of \eqref{eq:main-simple} started at $x \in \mathbb{R}$ such that
$X$ is a time-homogeneous strong Markov process spending zero time at $0$.
Let $Z = (H_\alpha(X_t))_{t \geq 0}$ started at $z = H_\alpha(x)$.

Then if $x \in(0,\infty)$,
then $Z$ is a Bessel process of dimension $\delta$ started at $z$ and if $x\in (-\infty,0)$,
then $-Z$ is a Bessel process of dimension $\delta$ started at $-z$. Eventually,
if $x = 0$ then there is $\theta\in[-1,1]$ such that
\ba
\Law(Z)=\frac{\theta+1}{2}\Law(\mathrm{BES}^\delta(0))+ \frac{\theta-1}{2}\Law(-\mathrm{BES}^\delta(0)).
\ea
\end{lemma}
\begin{proof}
Arguing as in Lemma \ref{lemma:from-positive-1-to-bessel} we conclude that $|Z|$ is a Bessel process of dimension
$\delta \in [2,\infty)$, started at $|z|$.
Therefore, $\P_z(Z_t\neq 0,\ t\in(0,\infty))=1$ and
$Z$ is either a Bessel process or the negative of a Bessel process, if $z > 0$ or $z < 0$, respectively.

For $z=0$, we have
\ba
1=\P_0( |Z_t|\neq 0,\ t\in(0,\infty))= \P_0( Z_t> 0,\ t\in(0,\infty)) + \P_0( Z_t < 0,\ t\in(0,\infty)).
\ea
and clearly setting
\ba
\theta:=  2\P_0( Z_t> 0,\ t\in(0,\infty))-1  \in[-1,1].
\ea
finishes the proof.
\end{proof}


\end{document}